\documentclass[11pt,leqno]{article}

\usepackage[T1]{fontenc}
\usepackage{amssymb}
\usepackage{amsmath}
\usepackage[title, toc]{appendix}
\usepackage{amsthm}  
\usepackage{tikz-cd}  
\usepackage{tikz,stackengine}  
\usepackage{enumerate} 
\usepackage{enumitem}  
\usepackage{pdfpages}  
\usepackage{hyperref}
\usepackage{wasysym}  
\usepackage{mathtools} 
\usepackage{stmaryrd} 
\usepackage{graphicx}
\usepackage{xcolor}
\usepackage{mathrsfs}
\usepackage{titlesec}
\usepackage{ytableau}
\usepackage{xcolor}
\hypersetup{
    colorlinks,
    linkcolor={blue!50!black},
    citecolor={blue!50!black},
    urlcolor={blue!80!black}
}
\DeclareSymbolFont{extraitalic} {U}{zavm}{m}{it}
\DeclareMathSymbol{\stigma}{\mathord}{extraitalic}{168}

\newtheorem*{theorem*}{Theorem}
\newtheorem{theorem}{Theorem}[subsection]
\newtheorem{proposition}[theorem]{Proposition}

\newtheorem*{proposition*}{Proposition}

\newtheorem{lemma}[theorem]{Lemma}
\newtheorem{corollary}[theorem]{Corollary}
\newtheorem{conjecture}{Conjecture}[subsection]
\theoremstyle{remark}
\newtheoremstyle{boldremark}
{\dimexpr\topsep/2\relax} 
{\dimexpr\topsep/2\relax} 
{}          
{}          
{\bfseries} 
{.}         
{.5em}      
{}          
\theoremstyle{boldremark}\newtheorem{bremark}[theorem]{Remark}
\newtheoremstyle{bolddefinition}
{\dimexpr\topsep/2\relax} 
{\dimexpr\topsep/2\relax} 
{}          
{}          
{\bfseries} 
{.}         
{.5em}      
{}          
\theoremstyle{bolddefinition}\newtheorem{bdefinition}[theorem]{Definition}
\newtheorem*{conjecture*}{Conjecture}
\makeatletter
\newcommand*\leftdash{\rotatebox[origin=c]{-45}{$\dabar@\dabar@\dabar@$}}
\newcommand*\rightdash{\rotatebox[origin=c]{45}{$\dabar@\dabar@\dabar@$}}
\makeatother
\newcommand{\Fr}{\operatorname{Fr}}

\newcommand{\Qlbar}{\overline{\mathbb{Q}}_{\ell}}

\newcommand{\Fqbar}{\overline{\mathbb{F}}_q}
\newcommand{\Fq}{\mathbb{F}_q}

\newcommand{\GL}{\textrm{GL}}

\newcommand{\h}{\mathfrak{h}}

\newcommand{\g}{\mathfrak{g}}

\newcommand{\yy}{\mathcal{Y}}

\newcommand{\F}{\mathcal{F}}

\newcommand{\hc}{\mathfrak{hc}}
\newcommand{\cc}{\underline{\Qlbar}}

\newcommand{\DD}{\hat{\Delta}}
\newcommand{\dd}{\hat{\delta}}
\newcommand{\NN}{\hat{\nabla}}
\newcommand{\LL}{\hat{\mathcal{L}}}
\newcommand{\T}{\hat{\mathcal{T}}}
\newcommand{\IC}{\operatorname{IC}}
\newcommand{\Rep}{\operatorname{Rep}}
\newcommand{\Ho}{\operatorname{Ho}}
\newcommand{\pro}{\operatorname{pro}}
\newcommand{\M}{\mathscr{M}}
\newcommand{\cM}{\widehat{\mathscr{M}}}
\newcommand{\Tilt}{\operatorname{Tilt}}
\newcommand{\Perv}{\hat{\mathcal{P}}}

\newcommand{\ind}{\operatorname{Ind}}
\newcommand{\res}{\operatorname{Res}}
\newcommand{\SBim}{\mathbb{S}\!\operatorname{Bim}}
\newcommand{\HH}{\operatorname{HH}}
\newcommand{\Hom}{\operatorname{Hom}}
\newcommand{\Spec}{\operatorname{Spec}}
\newcommand{\prolim}{\text{``$\lim\limits_{\leftarrow}$''}}
\newcommand{\dirlim}[1]{\lim\limits_{\xrightarrow[#1]{}}}
\newcommand{\invlim}[1]{\lim\limits_{\xleftarrow[#1]{}}}
\newcommand{\tw}[1]{\langle #1 \rangle}
\newcommand{\D}{\mathbb{D}}

\newcommand{\Frinv}{\mathrel{\scalebox{0.7}{\reflectbox{\rotatebox[origin=c]{180}{$\operatorname{F}$}}}}}
\newcommand{\Av}{\operatorname{Av}}
\newcommand{\FT}{\operatorname{FT_{\psi}}}
\newcommand{\hHoL}{h\Ho\Lambda}

\newcommand{\Ad}{\operatorname{Ad}}
\newcommand{\sgn}{\operatorname{sgn}}

\titleformat{\subsection}[runin]
{\normalfont\bfseries}{\thesubsection}{1em}{}

\titleformat{\subsubsection}[runin]
{\normalfont\bfseries}{\thesubsection}{1em}{}
\newcommand{\Addresses}{{
  \bigskip
  \footnotesize

  R.~Bezrukavnikov, \textsc{Department of Mathematics, Massachusetts Institute of Technology,
    Cambridge, Massachusetts}\par\nopagebreak
  \textit{E-mail address}: \texttt{bezrukav@math.mit.edu}

  \medskip

  K.~Tolmachov, \textsc{University of Edinburgh,
    Edinburgh, United Kingdom}\par\nopagebreak
  \textit{E-mail address}: \texttt{tolmak@khtos.com}

}}

\begin{document} \title{Monodromic model for Khovanov-Rozansky homology}

\author{Roman Bezrukavnikov, Kostiantyn Tolmachov} \date{}
\maketitle \abstract{
We describe a new geometric model for the Hochschild cohomology of Soergel bimodules based on the monodromic Hecke category studied earlier by the first author and Yun. Moreover, we identify the objects representing individual Hochschild cohomology groups (for the zero and the top degree cohomology this reduces to an earlier result of Gorsky, Hogancamp, Mellit and Nakagane). These objects turn out to be closely related to explicit character sheaves corresponding to exterior powers of the reflection representation of the Weyl group. Applying the described functors to the images of braids in the Hecke category of type A we obtain a geometric description for Khovanov-Rozansky knot homology, essentially different from the one considered earlier by Webster and Williamson.}
\tableofcontents
\section{Introduction}

In \cite{khovanov2007triply}, Khovanov defined a triply-graded link invariant and proved that it coincides with the invariant defined in \cite{khovanov2008matrix} by Khovanov and Rozansky. We refer to this invariant as Khovanov-Rozansky link homology. Khovanov's construction uses the braid group action, defined by Rouquier, on the homotopy category of Soergel bimodules. For a link given as a closure of a braid, the invariant is computed by applying the Hochschild cohomology functor to the Rouquier complex of Soergel bimodules categorifying this braid.

The category of Soergel bimodules is a graded algebraic version of the Hecke category. The latter has several geometric incarnations.

Let $G$ be a connected reductive group defined over the finite field $\Fq$, and let $B \subset G$ be a split Borel subgroup. Let $D^b_m(B\backslash G/B)$ be the mixed derived category of sheaves on the stack $B\backslash G/B$.  This is a monoidal category we will refer to as the mixed equivariant Hecke category.

A geometric model for Hochschild homology based on the mixed equivariant version of the Hecke category was described by Webster and Williamson in a series of works \cite{webster2008geometric}, \cite{webster2009geometry}, \cite{webster2017geometric}. Applied to $G$ of type A, this gives a geometric model for Khovanov-Rozansky homology.

Let $G^{\vee}$ be the Langlands dual group, $B^{\vee} \subset G^{\vee}$ a split Borel subgroup, and $U^{\vee} \subset B^{\vee}$ its unipotent radical. In \cite{by}, the monodromic Hecke category was defined as a completion of the mixed equivariant category of sheaves $D_m^b(U^{\vee}\backslash G^{\vee}/U^{\vee})$ on the basic affine space, monodromic with respect to the right (equivalently, left) action of the maximal torus $T^{\vee} \subset B^{\vee}$. In loc. cit. it was also proved, generalizing the ideas of \cite{s}, \cite{bgs}, \cite{bg}, that the equivariant and the monodromic Hecke categories are monoidally equivalent, for an arbitrary Kac-Moody group $G$. This equivalence has many features of the Koszul duality, and is referred to as such.

In this paper, we describe a new geometric model for Hochschild cohomology of Soergel bimodules, based on the above mixed monodromic version of the Hecke category. This approach has several advantages. One is that we are able to identify objects representing individual Hochschild cohomology groups. This generalizes the result of \cite{ghmn}, where this was done for top and bottom Hochschild cohomology. We also express these objects as images of explicit character sheaves under the Harish-Chandra transform.

Our initial motivation comes from the conjectures of \cite{gorsky2016flag} and \cite{tolmachov2018towards}. Namely, in \cite{gorsky2016flag} it is conjectured that there is a pair of adjoint functors between a certain derived category of coherent sheaves on the flag Hilbert scheme and the homotopy category of Soergel bimodules. Loc. cit. also proposed a way to compute Khovanov-Rozansky homology inside the coherent category, based on the above conjecture. Individual Hochschild cohomology groups are predicted to be represented by the exterior powers of the tautological bundle on the flag Hilbert scheme. There are natural filtrations of the tautological bundle and its exterior powers, with subquotients isomorphic to line bundles. The conjecture also states that the corresponding line bundles are sent to the Rouquier complexes categorifying products of the so-called Jucys-Murphy braids.

On the other hand, it was observed in \cite{tolmachov2018towards} that the character sheaves we obtain in this paper also categorify the elementary symmetric polynomials in Jucys-Murphy braids, matching those appearing in the categories of coherent sheaves on the flag Hilbert scheme. It was also observed that the corresponding objects in the Hecke category are images of the exterior powers of a certain central sheaf on the affine flag variety under the (conjectural) flattening functor from the affine Hecke category to the finite Hecke category.

Thus, this paper gives further evidence towards the above conjectures.

\subsection{Hochschild homology and monodromic categories.} See Sections \ref{sec:khov-rozansky-homol}, \ref{sec:gener-compl-categ} for the precise definitions.

Let $G$ be a split connected reductive group defined over a finite field $\Fq$.

Let $W$ be its Weyl group. The $\ell$-adic Tate module $V_T \simeq \operatorname H_1(T,\Qlbar)$ of the maximal torus $T$ carries a representation of $W$. The module $V_T$ is a vector space over $\Qlbar$ of dimension equal to the rank of $G$.

Recall (see e.g. \cite{soergel2007kazhdan}, \cite{ew}, \cite{elias2020introduction}) that to $V_T$, considered as a $W$-represenation, one can attach the category of Soergel bimodules, denoted by $\SBim$ in this introduction. In fact, the category of Soergel bimodules can be attached to any Coxeter group and its representation (not necessarily over a field of characteristic 0) satisfying certain technical condition, trivially satisfied in our case. We shall not use this more general setting.

$\SBim$ is a full additive subcategory of the category of graded $R$-bimodules, where $R = \operatorname{Sym}V_T$ with $V_T$ placed in degree 2.

Recall from loc. cit. that the isomorphism classes of indecomposable Soergel bimodules, up to a grading shift, are labeled by the elements of the Weyl group. Let $B_w$ be a representative of a class labeled by $w$.

In \cite{by}, a family of indecomposable free-monodromic tilting objects in the completion $\cM$ of the category $D^b_m(U\backslash G/U)$ was constructed. Let $\T_w\in\cM$ be an indecomposable free-monodromic tilting object labeled by $w \in W$, constructed in loc. cit. There is a functor $\Lambda:\SBim \to \operatorname{Tilt}$, where $\operatorname{Tilt}$ is the additive category generated by $\T_w$ and their Tate twists. This functor is inverse to the functor $\mathbb{V}$ constructed in loc. cit. We have $\Lambda(B_w) = \T_w$.

Let $h\Ho(\Lambda):\Ho(\SBim) \to \cM$ denote the extension of the functor $\Lambda$ to the homotopy category of Soergel bimodules (see Section \ref{sec:comp-with-soerg-4}).

For two objects $\F, \mathcal{G}$ in the mixed derived category on the stack $X$ over $\Fq$, we write $\Hom(\F,\mathcal{G})$ for the $\Hom$-space in the derived category of sheaves on $X\otimes_{\operatorname{Spec}\Fq}{\operatorname{Spec}{\Fqbar}}$ with the Frobenius action. For a Frobenius module $M$, let $M^f$ be the union of all its finite dimensional submodules, and let $M^{\Frinv}$ be the module with the inverse Frobenius action. If the eigenvalues of the Frobenius on $M$ are integral, let $\underline{M}$ be the corresponding graded $\Qlbar$-vector space.

For a finitely-generated $\Fr$-module $M$ with intergral weights over the power series ring $\Qlbar[[V_T]]$, $\underline{M}^f$ is naturally a finitely-generated graded module over the polynomial ring $\operatorname{Sym}[V_T]$. 

Consider the adjoint action of $T$ on $U\backslash G/U$ and let $p$ be the projection $p: U\backslash G/U \to \dfrac{U \backslash G/U}{T}$. Our first result is the following (see Propositions \ref{sec:comp-with-soerg-4} and \ref{sec:comp-with-soerg}):
\begin{proposition} There is an isomorphism of functors on $\SBim$
  \[ \operatorname{HH}^{\bullet}(M) \dot{=} \underline{\operatorname{Hom}^{\bullet}(\mathbb{K}, \Lambda(M))^{f,\Frinv}},
  \] where $\mathbb{K} = p^*p_!\T_{w_0}$, $w_{0}$ is the longest element of $W$ and $\dot{=}$ denotes an isomorphism up to shifts.
\end{proposition}
Note that inversion of the Frobenius action $\Frinv$ comes from the fact that the Koszul duality for Hecke categories from \cite{by} sends the Tate twist to the composition of the inverse twist and homological shift.

Informally, we say that $\mathbb{K}$ represents the Hochschild cohomology functor on Soergel bimodules. Our main observation here is that, for the purposes of computing Hochschild cohomology of Soergel bimodules, the Koszul resolution $K^{\bullet}$ of the diagonal bimodule over the polynomial ring can be replaced by a complex $K_{\mathbb{S}}:=K^{\bullet}\otimes_{R \otimes R}B_{w_0}$ of Soergel bimodules. We then observe that the complex of tilting sheaves $h\Ho(\Lambda)(K_{\mathbb{S}})$ can be geometrically described as the averaging $\mathbb{K} = p^*p_!\T_{w_0}$ of $\T_{w_0}$ with respect to the adjoint $T$-action.

\subsection{Tilting sheaves and the Radon transform.} Summarizing the discussion above, we get that to recover individual Hochschild cohomology groups $\operatorname{HH}^k(B_w)$ one has to compute the cohomology of the complex $\Hom(p^*p_!\T_{w_0}, \T_w)$. To do this, it turns out to be natural to consider a new t-structure on the monodromic Hecke category, which we now describe.

Recall that the functor of left convolution with $\NN_{w_0}$ is called the \emph{Radon transform} (see Example 4.1.2 of \cite{yun2009weights} for an explanation of this terminology).

In the setting of sheaves on $U\backslash G/B$ (considered as a left module category over the monodromic Hecke category) it was proved in \cite{beilinson2004tilting} that the Radon transform functor takes tilting perverse sheaves to injective perverse sheaves on $U\backslash G/B$. In our setting the above statement can be reformulated as follows: tilting complexes are injective with respect to the pullbacks of the perverse t-structure along the autoequivalence $(\NN_{w_0}\star -)$. Let $\mathcal{H}^{-k}_{w_0}(-)$ stand for the cohomology functor with respect to this shifted t-structure. We denote
\[
  \mathfrak{E}_k = \mathcal{H}^{-k}_{w_0}(\mathbb{K}).
\]
Our second observation is as follows (see Theorem \ref{sec:comp-with-soerg-2}):
\begin{theorem}\label{sec:tilt-sheav-radon} For $M \in \SBim$, there is a functorial isomorphism
\[ \operatorname{HH}^k(M) \dot{=} \underline{\Hom(\mathfrak{E}_k, \Lambda(M))^{f,\Frinv}}.
\] Here  $\dot{=}$ denotes an isomorphism up to shifts.
\end{theorem}
\begin{bremark} For technical convenience, we work not with the standard t-structure on the completed category defined in \cite{by}, but with the dual one, see Section \ref{sec:gener-compl-categ}.
\end{bremark}

\subsection{Rouquier complexes and Khovanov-Rozansky link invariant.} For the purposes of this subsection, let $G = \operatorname{GL}_n$. Let $\operatorname{Br}_n$ stand for the braid group on $n$ strands. In \cite{rouquier2004categorification}, the complex $F_{\beta}$ in the homotopy category of Soergel bimodules attached to any $\beta \in \operatorname{Br}_n$ is constructed. Moreover, the tensor product with $F_{\beta}$ defines the categorical $\operatorname{Br}_n$-action on the homotopy category of Soergel bimodules. In \cite{khovanov2007triply}, the invariant
\[ \operatorname{HHH}^k(\beta) = \operatorname{H}^{\bullet}(\operatorname{HH}^k(F_{\beta}))
\] is shown, up to grading shifts, to depend only on the link closure of the braid $\beta$. Note that the invariant above, called the Khovanov-Rozansky link invariant, is defined as the Hochschild cohomology functor applied to a complex of Soergel bimodules term by term. For $w \in W$, denote also by $\sigma_w$ its lift to the $\operatorname{Br}_n$. Arbitrary Rouquier complex can be expressed as a tensor product of complexes $F_{\sigma_w}$ (denoted $\Delta_w$ in the main text) and their monoidal inverses (denoted $\nabla_w$, respectively). The functor $h\Ho(\Lambda)$ sends $F_{\sigma_w}$ to the standard pro-unipotent object $\hat{\Delta}_w$ and its inverse to the costandard pro-unipotent object $\hat{\nabla}_w$, see Section \ref{sec:stand-cost-pro}. The complex $F_\beta$ corresponding to an arbitrary braid is sent to the corresponding convolution of standard and costandard objects.

The discussion above implies that (see Theorem \ref{sec:comp-with-soerg-2}):
\begin{theorem}\label{sec:rouq-compl-khov} There is a functorial isomorphism
\[ \operatorname{HHH}^k(\beta) \dot{=} \underline{\Hom(\mathfrak{E}_k, h\Ho(\Lambda)(F_{\beta}))^{f,\Frinv}}.
\]
\end{theorem}

Theorems \ref{sec:tilt-sheav-radon} and~\ref{sec:rouq-compl-khov} can be formulated entirely inside the homotopy category of Soergel bimodules. Namely, there is a t-structure on $\Ho(\SBim)$ with cohomology functor denoted by $\mathcal{H}_{\mathbb{S},w_0}^{\bullet}(-)$. The functor $h\Ho\Lambda$ is t-exact with respect to it and the shifted t-structure on $\cM$. We denote
\[
  \mathbb{E}_k = \mathcal{H}^{-k}_{\mathbb{S},w_0}(K_{\mathbb{S}}).
\]
See Proposition \ref{sec:comp-with-soerg-4} and Corollary \ref{sec:comp-with-soerg-7}.
\begin{corollary}  There are functorial isomorphisms
\[ \operatorname{HH}^k(M) \dot{=} \Hom(\mathbb{E}_k, M),
\]
\[ \operatorname{HHH}^k(\beta) \dot{=} \Hom(\mathbb{E}_k, F_{\beta}).
\] for $M \in \SBim, \beta \in \operatorname{Br}_n$.
\end{corollary} In the Appendix, we compute some examples in types $\operatorname{A}_1, \operatorname{A}_2$.

\subsection{Character sheaves and Hochschild cohomology.} Finally, we relate the above shifted perverse cohomology groups with character sheaves. For a scheme $X$ with an action of an algebraic group $H$ we write $\Av_H$ for the averaging with compact support, i.e. the functor $\pi_!$ for the canonical projection $X \to X/H$.

We have a !-version of the Harish-Chandra functor \[\hc_!:D^b_m(G/_{\Ad}G)\to D^b_m\left(\dfrac{U \backslash G/U}{T}\right),\] that is given by the averaging $\Av_U$ with respect to the left (or right) multiplication action of $U$ on $G$. The derived category of unipotent character sheaves $D\mathcal{CS}$ is the full triangulated subcategory consisting of objects $\F \in D^b_m(G/_{\Ad}G)$ such that $\hc_!(\F)$ is $T$-monodromic with unipotent monodromy, see \cite{lcs1}, \cite{mirkovic1988characteristic}. Recall that, according to \cite{bfo} in characteristic zero and to \cite{chenyd} in the $\ell$-adic setting, $\hc_!$ is a t-exact functor from $D\mathcal{CS}$ (with the perverse t-structure) to the $\Ad_T$-equivariant Hecke category $D^b_m\left(\dfrac{U \backslash G/U}{T}\right)$ (with the shifted perverse t-structure described above). To compute the shifted perverse cohomology of $\mathbb{K}$ in terms of character sheaves, we find an $\Ad_G$-equivariant complex on $G$ that is mapped to $\mathbb{K}$ by the Harish-Chandra functor. 

Let $\mathcal{N}_G \subset G$ stand for the unipotent variety of $G$, and let $j:\mathcal{N}_G^{reg} \to \mathcal{N}_G$ be the embedding of the open subset of regular unipotent elements. Let $Z(G)$ stand for the center of $G$, and let $\Av_{Z(G)}$ stand for the averaging functor (with compact support) with respect to the (trivial) adjoint action of $Z(G)$. We prove the following result, see Theorem \ref{sec:whittaker-averaging-5}.

For a stack $X$, let $\cc_X$ stand for the constant rank one  $\Qlbar$-sheaf on $X$. 

Let $-\star-$ stand for the operation of convolution with compact support on $D_m^b(G/_{\Ad}G)$ and $\cM$, see Sections \ref{sec:harish-chandra-funct-1} and \ref{sec:convolution}, respectively. 

Let $\Xi = \Av_{Z(G)}j_*\underline{\Qlbar}_{\mathcal{N}_{G}^{reg}}$.
\begin{theorem}
\label{sec:char-sheav-hochsch}
  Denote by $\Xi\star\dd_G$ the projection of the sheaf $\Xi$ to the derived category of character sheaves. Then $p^*\hc_!(\Xi\star\dd_G) \dot{=} \mathbb{K}$. Here $\dot{=}$ denotes an isomorphism up to shifts and twists.
\end{theorem}

Here by the projection to the derived category of character sheaves we mean a convolution with a certain pro-object $\dd_G$ in the category $D\mathcal{CS}$, which serves as a unit in the category of unipotent character sheaves. The existence of such an object is implicit in the work \cite{bfo} in the characteristic 0 setting, where the Drinfeld center of the (abelian) Hecke category is shown to be equivalent to the category of character sheaves. Since the unit is naturally a central object, it must come as an image of some (pro-)object on $G$. We construct this object in the $\ell$-adic setting using the result of Chen from \cite{chen}.

We describe the object $\dd_G$ in more detail. Let $G^{rss}$ be the subset of regular semisimple elements. For a local system $E$ on $T/W$, let $E^G$ be its pullback under the map $G^{rss}/_{\Ad}G \to T/W$. $E^G$ can be viewed as a $G$-equivariant local system on $G^{rss}$. We identify the group algebra $\Qlbar[\pi_1^{\ell}(T)]$ with $S = \operatorname{Sym}V_T$. Here $\pi_1^{\ell}$ stands for the $\ell$-adic tame quotient of the fundamental group. Let $S^W$ stand for the subring of $W$-invariant polynomials in $S$, and let $\mathfrak{m}_W\subset S^W$ be an ideal generated by homogeneous $W$-invariant polynomials of positive degree.  Let $\varepsilon_n$ be a $W$-equivariant local system on $T$ corresponding to a representation $S/S\mathfrak{m}_W^n \simeq S\otimes_{S^W}S^W/\mathfrak{m}_W^n$ of $\pi_1^{\ell}(T)$. Let $\mathcal{E}_n = \IC(\varepsilon_{n}^G)$, where $\IC$ stands for the Goresky-MacPherson extension of the perversely shifted local system. For a projective system of objects $(A_n)$ write formally $\prolim A_n$ for the corresponding pro-object. We have the following description of $\dd_G$, see Corollary \ref{centralunit}. 

\begin{proposition}
  For any monodromic complex $\F\in\cM$ we have $\invlim{n}(\F\star p^*\hc_!(\mathcal{E}_n)) \dot{=} \F$, so that  $\prolim p^*\hc_!(\mathcal{E}_n) \dot{=} \DD_{e} \simeq \NN_{e}$ is the unit in the category $\cM$, isomorphic to a standard (and costandard) object $\DD_{e}\simeq \NN_{e}$, for the unit $e \in W$ (up to a shift and Tate twist).  
\end{proposition}

Since the functor $\hc_!$ can easily be seen to be faithful (see Lemma~\ref{sec:four-deligne-transf-5}), it can be deduced that $\dd_G = \prolim\mathcal{E}_n$ serves as a pro-unit in the category $D\mathcal{CS}$.
\begin{bremark}
  One can consider $\dd_G$ as an object in the pro-completion
  of the category of unipotent character sheaves admitting an embedding
  in the completion of the Hecke category defined in \cite{by}. We do not
  develop this topic in this paper.
\end{bremark}

We also compute the perverse cohomology of $\Xi = \Av_{Z(G)}j_*\cc_{\mathcal{N}_G^{reg}}$ in terms of character sheaves corresponding to the exterior powers of the representation of the Weyl group on $\mathfrak{t} := H^1(T, \Qlbar)$. Namely, let $\mathfrak{W}_k$ be a summand of Springer sheaf corresponding to the representation $\wedge^k\mathfrak{t}$ of $W$, see Section \ref{sec:four-deligne-transf-4}. These are semisimple perverse sheaves supported on the unipotent variety of $G$. We have the following result, see Theorem \ref{sec:whittaker-averaging-1}.
\begin{theorem}
  There is an isomorphism
  \[\mathcal{H}^k(\Xi) \dot{=} \mathfrak{W}_{n-k}.\]
\end{theorem}
This, together with Theorem \ref{sec:char-sheav-hochsch} and t-exactness properties of the functor $\hc_!$ mentioned above, gives an explicit description of the objects representing the Hochschild cohomology of Soergel bimodules in terms of character sheaves, see Theorem \ref{sec:char-sheav-full-3}.
\begin{theorem}
 There is an isomorphism  
       \[
         \mathfrak{E}_k \dot{=} p^{\dagger}\hc_!(\mathfrak{W}_k\star\dd_G).
       \]
\end{theorem}
For example, one immediately recovers the result proved in \cite{ghmn} for type A, that the highest degree Hochschild cohomology is represented by the full twist braid, see Corollary \ref{sec:whittaker-averaging-6}.

We also give a different description of the pro-object that is sent to the full twist braid by the Harish-Chandra functor. Namely, let $\sgn\otimes\varepsilon_n$ be the $W$-equivariant local system corresponding to the mo\-dule $\operatorname{sgn}_W\otimes_{\Qlbar}S/S\mathfrak{m}_W^n$, where $\operatorname{sgn}_W$ stands for the sign representaion of $W$. Let $\mathcal{E}_n^{\operatorname{sgn}} = \IC((\sgn\otimes\varepsilon_{n})^G)$. We obtain the following result, see Corollary \ref{sec:char-sheav-full-2}.

\begin{proposition}
  We have the following isomorphisms, up to shifts and Tate twists,
  \[
    \prolim p^*\hc_!(\mathcal{E}_n^{\sgn}) \dot{=} \DD_{w_0}\star\DD_{w_0},
  \]
  \[
    \mathcal{E}_n^{\sgn}\dot{=} \cc_{\mathcal{N}_G}\star\mathcal{E}_n.
  \]
\end{proposition}
Thus the pro-unit and the full twist objects are minimal extensions 
of local systems on $G^{rss}$ which differ by a sign twist.

We also obtain a filtration of the object $\mathfrak{E}_k$ representing the $k$th cohomology by the products of Jucys-Murphy braids, see Section \ref{sec:jm-filtr}, especially Theorem~\ref{sec:jucys-murphy-filtr-1}. See also \cite{tolmachov2018towards}. Together with the discussion in loc. cit., this gives a new piece of evidence to the Gorsky-Negu\c{t}-Rasmussen conjecture of \cite{gorsky2016flag}, which we now explain.

\subsection{Relation to the conjectures in \cite{gorsky2016flag}.}

Recall that in loc. cit. a dg version of a flag Hilbert scheme $\operatorname{FHilb_n^{dg}}(\mathbb{C})$ is defined, together with a tautological vector bundle $\mathcal{T}_n$ of rank $n$ on it. This scheme parametrizes full flags of ideals containing an ideal of codimension $n$ in $\mathbb{C}[x,y]$, with successive quotients supported on the line $y = 0$ (hence $\mathbb{C}$ in the notation).

$\mathcal{T}_n$ has a filtration whose subquotients are line bundles denoted $\mathcal{L}_k, k = 1, \dots, n$ in loc. cit.

Let $\Ho(\SBim_n)$ stand for the bounded homotopy category of Soergel bimodules associated to $G = \GL_n$ as above. Let $s_1, \dots, s_{n-1}$ be the set of simple reflections of $S_n$, and write $\sigma_i$ for the lift of $s_i$ to the braid group. Recall that Jucys-Murphy braids $\mathfrak{j}_0, \dots, \mathfrak{j}_{n-1}$ are defined as \[\mathfrak{j}_0 = 1, \mathfrak{j}_{k} = \sigma_k\mathfrak{j}_{k-1}\sigma_k.\] Conjecture 1.1 of loc. cit. states, in particular,
\begin{conjecture}[see \cite{gorsky2016flag}, Conjecture 1.1 in Section 3.6]
  \label{sec:relat-conj-citeg} There is a pair of adjoint functors
  \[ \iota_*:\Ho(\SBim_n) \leftrightarrows D^b(\operatorname{Coh}_{\mathbb{C}^*\times\mathbb{C}^*}(\operatorname{FHilb}^{dg}_n)):\iota^*,
 \] with $\iota^*(\mathcal{L}_k)) \simeq \mathfrak{J}_{k-1}$. Here $\mathfrak{J}_k$ stands for the Rouquier complex $h\Ho\Lambda(F_{\mathfrak{j}_k})$ corresponding to the $k$th Jucys-Murphy braid.

\end{conjecture} Moreover, in loc. cit. it is predicted that the Khovanov-Rozansky homology can be computed on the coherent side as follows.
\begin{conjecture}[see \cite{gorsky2016flag}, Conjecture 1.1 of Introduction and Conjecture 3.8] There is an isomorphism of bigraded vector spaces
   \[ \operatorname{HHH}^{k}(F_\beta)\simeq \operatorname{HH}^0(\iota^*(\wedge^k\mathcal{T}_n^{\vee})\otimes F_{\beta})\simeq \int \iota_*(F_\beta) \otimes \wedge^k\mathcal{T}_n^{\vee},
   \] where $\int$ denotes the pushforward to the point in the coherent equivariant category.
\end{conjecture} The discussion in loc. cit. also implies that, assuming the conjectures, one can rewrite
\[ \operatorname{HH}^0(\iota^*(\wedge^k\mathcal{T}_n^{\vee})\otimes F_{\beta}) = \Hom(\iota^*(\wedge^k\mathcal{T}_n), F_{\beta}).
\] Thus the objects $\mathfrak{E}_k$ representing Hochschild cohomology, constructed in this article, can be identified with (conjectural) images of the exterior powers of the tautological bundle $\mathcal{T}_n$ under the functor $\iota^*$. For this to match with the Conjecture \ref{sec:relat-conj-citeg}, these objects must have filtrations (in the sense of triangulated categories) by the objects corresponding to products of Jucys-Murphy braids, matching the filtrations of the corresponding vector bundles by line bundles. The existence of such filtrations is explained in Section \ref{sec:jm-filtr}.

\subsection{Organization of the paper.} In Section \ref{sec:khov-rozansky-homol} we recall the definition of the Hochschild cohomology of Soergel bimodules and fix the required notation. We also recall the t-structure and duality on the homotopy category of Soergel bimodules defined in \cite{achar2019mixed}.

In Section \ref{sec:gener-compl-categ} we define the duality functor on the completed category of \cite{by}.

In Section \ref{sec:monodr-hecke-categ} we recall the facts about the monodromic Hecke category we will use, state the relations of our constructions with the homotopy category of Soergel bimodules and define the complex $\mathbb{K}$, representing the Hochschild cohomology.

In Section \ref{sec:char-sheav-harish} we recall the definition of the Harish-Chandra functor and prove the main results of this paper, relating Hochschild cohomology of the Hecke category to character sheaves.

In Section \ref{sec:jm-filtr} we construct the filtrations of the objects representing Hochschild cohomology by the objects corresponding to products of Jucys-Murphy braids, and relate our results to the decategorified computations with Jones-Ocneanu traces. 

In the Appendix we illustrate our result by doing a computation in the homotopy categories of Soergel bimodules of types $\operatorname{A}_1, \operatorname{A}_2$.
\subsection{Acknowledgments.} We would like to thank the anonymous referees for their helpful suggestions on improving this text. We thank Alexander Braverman for useful discussions, especially for suggesting an idea that helped us establish Theorem \ref{sec:whittaker-averaging-1} in its present generality. The first author would like to thank Tanmay Deshpande for related discussions. He is partially supported by NSF grant DMS-2101507. The second author would like to thank Eugene Gorsky for the constant interest in this work, encouragement, many very helpful discussions, and comments on a draft of this paper; and Stefan Dawydiak for reading of the draft and suggestions to improve the presentation. He would also like to acknowledge the role of the WARTHOG workshop ``Knot homologies, Hilbert schemes, and Cherednik algebras'' of 2016 and AIM workshop ``Categorified Hecke algebras, link homology, and Hilbert schemes'' of 2018, where he learned about the problem and many related things, and to thank their participants and organizers.

This work was done while the second author was a Postdoctoral Fellow at the Perimeter Institute of Theoretical Physics. Perimeter Institute is supported in part by the Government of Canada through the Department of Innovation, Science and Economic Development Canada and by the Province of Ontario through the Ministry of Colleges and Universities.
\section{Khovanov-Rozansky homology}
\label{sec:khov-rozansky-homol}

\subsection{Soergel bimodules.} See \cite{ew}, \cite{ghmn} for an introduction to Soergel bimodules and the discussion of some relevant homological algebra.

Let $\mathfrak{g}_0$ be a reductive Lie algebra over an algebraically closed field $k$ of characteristic 0. In the sequel we will be mostly interested in $k = \Qlbar$.

Let $\mathfrak{h}_0$ be a Cartan subalgebra of $\g_0$.

Let $W$ be its Weyl group, with the set of simple reflections $\Sigma$. Let $l: W \to \mathbb{Z}_{\geq 0}$ be the length function with respect to $\Sigma$, and let $w_0 \in W$ be the longest element.

$W$ acts on $\h_0$ by reflections with respect to the dual collections $(\alpha_s)_{s \in \Sigma} \subset \h_0^\vee, (\alpha^{\vee}_s)_{s \in \Sigma} \subset \h_0$. Let $R = \operatorname{Sym}_k(\h_0^\vee)$. We turn $R$ into a graded ring by setting $\deg \h_0^\vee = 2$. Let $R-\operatorname{mod}-R$ be the category of finitely generated graded $R$-bimodules. This is a monoidal category, with the symmetric product given by $\otimes_{R} =: \otimes$. We will frequently omit the tensor symbol from notations, writing $AB$ in place of $A\otimes_R B$. For a graded $R$-bimodule $M = \oplus_i M^i$, write $M(r)^i = M^{i+r}$ for its grading shift by $r$.

When $\mathcal{C}$ is an additive category, we write $\Ho(\mathcal{C})$ for the bounded homotopy category of $\mathcal{C}$. For all triangulated categories we consider, we let $[k]$ denote the homological shift by $k$.

For $s \in \Sigma$, let $R^s$ denote the subring of $s$-invariant elements in $R$. Let $B_s = R \otimes_{R^s} R(1)$. Recall that the category $\SBim_R(W)$ of Soergel bimodules is defined as the smallest full additive monoidal subcategory of $R-\operatorname{mod}-R$ closed under taking direct summands and grading shifts, containing the bimodules $R, B_s, s \in \Sigma$. By \cite{soergel2007kazhdan}, isomorphism classes of indecomposable Soergel bimodules, up to grading shifts, are labeled by the elements of $W$. Let $B_w$ be the indecomposable summand of $B_{s_{i_1}}\otimes \dots \otimes B_{s_{i_q}}$ for a reduced expression $w = s_{i_1}\dots s_{i_q}$ whose shifts do not appear as summands in the products $B_{t_{j_1}}\otimes \dots \otimes B_{t_{j_r}}$ for any reduced expression $v = t_{j_1}\dots t_{j_r}, v < w,$ where $<$ denotes the Bruhat order with respect to $\Sigma$.

For two graded bimodules $M, N$, we will write $\hom(M,N)$ for the space of morphisms in the category of \emph{graded} bimodules, and denote \[\Hom(M,N) = \oplus_{i \in \mathbb{Z}}\hom(M, N(i)),\] the graded $\Hom$ space between $M$ and $N$ as ungraded bimodules.

\begin{bremark}
  \label{RW} It is easy to see that $R\otimes_kR$ action on any Soergel bimodule factors through the $R\otimes_{R^W}R$ action, where $R^W$ stands for the ring of symmetric polynomials. It is also known that $B_{w_0} \simeq R\otimes_{R^W}R(l(w_0))$, and so $\operatorname{End}(B_{w_0})\simeq R\otimes_{R^W}R$ as a graded algebra.

\end{bremark}
\subsection{Rouquier complexes.} Recall that for $s\in\Sigma$, the Rouquier complexes in $\Ho(\SBim_R(W))$ corresponding to simple reflections are defined as
\[ \Delta_s = \underline{B_s} \to R(1), \nabla_s = R(-1) \to \underline{B_s}.
\] Here underline marks the $0$th cohomological degree in the complex.

For $w \in W$ with a reduced expression $w = s_{i_1}\dots s_{i_k}$, Rouquier complexes corresponding to $w$ are defined as
\[ \Delta_w = \Delta_{s_{i_1}}\dots\Delta_{s_{i_k}}, \nabla_w = \nabla_{s_{i_1}} \dots \nabla_{s_{i_k}}.
\] They do not depend on the choice of the reduced expression, up to an isomorphism in $\Ho(\SBim_R(W))$, by a result of \cite{rouquier2004categorification}.

Tensor product with Rouquier complexes defines braid group action on $\Ho(\SBim_R(W))$, both on the left and on the right. Recall that the Artin-Tits braid group $\operatorname{Br}(W)$ can be presented with generators $\sigma_w, w \in W$, and relations $\sigma_v\sigma_w=\sigma_{vw}$ if $l(v)+l(w) = l(vw)$. On the other hand, we have the following isomorphisms in $\Ho(\SBim_R(W))$:
\[ \Delta_v\Delta_w \simeq \Delta_{vw}, \text{ when } l(v)+l(w) = l(vw),
\]
\[ \nabla_v\nabla_w \simeq \nabla_{vw}, \text{ when } l(v)+l(w) = l(vw),
\]
\[ \Delta_v\nabla_{v^{-1}} \simeq R.
\]
For an element $\beta \in \operatorname{Br}(W)$ written as a product of generators $\sigma_w, (\sigma_w)^{-1}$, $w \in W,$ let $F_\beta$ be the corresponding tensor product of the complexes $\Delta_w, \nabla_{w^{-1}}$, respectively. It was shown in \cite{rouquier2004categorification} that the complex $F_\beta$ is well-defined up to a canonical isomorphism in $\Ho(\SBim_R(W))$.
\subsection{t-structure and duality.}
\label{sec:t-structure-duality} In \cite{achar2019mixed} a t-structure and duality on the homotopy category $\Ho(\SBim_R(W))$ are defined. They correspond to the perverse t-structure and duality on the geometric Hecke category, as will be described below. For now, we record some of their properties.

The t-structure, which we denote $(\prescript{p}{}\SBim_R(W)^{\leq 0}, \prescript{p}{}\SBim_R(W)^{\geq 0})$, was defined in loc. cit. by gluing from t-structures on certain categories $\mathcal{D}_{\{w\}}, w \in W,$ each being equivalent to the homotopy category of free graded $R$-modules.

In more detail, let $\SBim_R(W)_{\leq w}$ be the minimal additive monoidal subcategory of $\SBim_R(W)$ containing all grading shifts of bimodules $B_v, v \leq w$ and closed under direct summands. Similarly, let $\SBim_R(W)_{< v}$ be the minimal additive monoidal subcategory of $\SBim_R(W)$ containing all grading shifts of bimodules $B_v, v < w$ and closed under taking direct summands. Let $\mathcal{D}_{\{w\}}$ be the Verdier quotient category $\Ho(\SBim_R(W)_{\leq w})/\Ho(\SBim_R(W)_{< w})$. The functors $i^*_w, i^!_w:\Ho(\SBim_R(W)) \to \mathcal{D}_{\{w\}}$ were constructed, and the categories $\mathcal{D}_{\{w\}}$ were proved in loc. cit. to be equivalent to the bounded derived category of ${D}^b(R-\operatorname{mod})$ of finitely generated graded $R$-modules. Let $({D}^{\leq 0}, {D}^{\geq 0})$ stand for the standard t-structure on the latter category, and let $(\mathcal{D}_{\{w\}}^{\leq 0}, \mathcal{D}_{\{w\}}^{\geq 0})$ stand for the t-structure on $\mathcal{D}_{\{w\}}$ corresponding to $({D}^{\leq 0}, {D}^{\geq 0})$ under the above equivalence. Now define the t-structure on $\Ho(\SBim_R(W))$ by gluing as follows:
\[ \prescript{p}{}{}\SBim_R(W)^{\leq 0} = \{\F \in \Ho(\SBim_R(W)),\forall w\ i_w^*\F \in \mathcal{D}_{\{w\}}^{\leq 0}\},
  \]
   \[ \prescript{p}{}{}\SBim_R(W)^{\geq 0} = \{\F \in \Ho(\SBim_R(W)),\forall w\ i_w^!\F \in \mathcal{D}_{\{w\}}^{\geq 0}\}.
 \]

Note that this definition uses the t-structure on $\mathcal{D}_{\{w\}}$ that is Koszul dual to the one used in \cite{achar2019mixed}.
 
This construction is a direct analogue of the construction of the perverse t-structure on a stratified space. Objects $B_w(m), \Delta_w(m), \nabla_w(m)$ are in the heart of this t-structure for any $w \in W, m \in \mathbb{Z}$. The category $(\prescript{p}{}\SBim_R(W))^{\leq 0}$ is generated under extensions by the objects $\Delta[n](m)$ with $n \geq 0$, see Lemma 7.5 of \cite{achar2019mixed}.

Recall that on the category $\SBim_R(W)$ the duality functor $M \mapsto M^{\vee}$ is defined as the usual duality in $R-\operatorname{mod}$ with respect to either left or right module structure. See, e.g., \cite{ghmn} for the detailed discussion. We have $B_w^{\vee} \simeq B_w, \nabla_w^{\vee} \simeq \Delta_w$.
\begin{bremark}
  \label{sec:t-structure-duality-2}
Note that the t-structure above is \emph{not} self-dual with respect to this duality. For example, the complex $ R(-2) \xrightarrow{\alpha_s} \underline{R} $ is in the heart, but its dual complex $ \underline{R} \xrightarrow{\alpha_s} R(2) $ is not. We will, in fact, need the dual t-structure $\left(\left(\prescript{p}{}\SBim_R(W)^{\geq 0}\right)^{\vee}, \left(\prescript{p}{}\SBim_R(W)^{\leq 0}\right)^{\vee}\right)$ in what follows.
\end{bremark}
\begin{bdefinition}
  \label{sec:t-structure-duality-1} Define the shifted t-structure $\left(\SBim_R(W)_{w_0}^{\geq 0},\SBim_R(W)_{w_0}^{\leq 0}\right)$ on $\Ho(\SBim_R(W))$ as
 \[ \SBim_R(W)_{w_0}^{\geq 0} = \Delta_{w_0} \star \left(\prescript{p}{}\SBim_R(W)^{\leq 0}\right)^{\vee}
 \] and
 \[ \SBim_R(W)_{w_0}^{\leq 0}= \Delta_{w_0} \star \left(\prescript{p}{}\SBim_R(W)^{\geq 0}\right)^{\vee}.
 \]
 Let $\mathcal{H}_{\mathbb{S},w_0}^{\bullet}$ stand for the cohomology functor with respect to this t-structure.
\end{bdefinition}

\subsection{Hochschild cohomology.}
\label{sec:hochsch-cohom} For any $M \in R-\operatorname{mod}-R$, let $\HH^i(M) = \operatorname{Ext}^i(R, M)$ stand for the Hochschild cohomology functor. It can be computed as follows: for $s \in \Sigma$, let \[K_s = {R \otimes_{k} R}(-2) \xrightarrow{\alpha_s\otimes 1 - 1 \otimes \alpha_s} \underline{R \otimes_{k}R} \] so that $K^{\bullet} = \bigotimes_{s\in\Sigma}K_s$ (tensor product over $R \otimes_{k} R$) is the Koszul resolution of $R$ as an $R \otimes_{k} R$-bimodule. Here underline marks the $0$th cohomological degree in the two-term complex.

Then we have
\[ \HH^i(M) \simeq \operatorname{H}^i(\underline{\Hom}(K^{\bullet},M)),
\] where $\underline{\Hom}$ stands for the complex of $\Hom$-spaces.

It would be important to us to rewrite the above expression for $M \in \SBim_R(W)$. Consider the modified complex
\[K_{\mathbb{S}} = K^{\bullet}\otimes_{R \otimes_{k} R} B_{w_0}.\]
\begin{lemma} For $M \in \SBim_R(W)$, there is a natural isomorphism
 \[ \HH^i(M)(-l(w_0)) \simeq \operatorname{H}^i(\underline{\Hom}(K_{\mathbb{S}}, M)).
\]
\end{lemma}
\begin{proof}

  In view of the Remark \ref{RW}, we have \[\underline{\Hom}(K^{\bullet}, M) \simeq \underline{\Hom}(K_{\mathbb{S}}, M(l(w_0))),\] so that 
  \[ \HH^i(M) = \operatorname{H}^i(\underline{\Hom}(K^{\bullet}, M)) = \operatorname{H}^i(\underline{\Hom}(K_{\mathbb{S}}, M(l(w_0)))).
  \]

\end{proof}

The functors $\HH^i$, considered as functors from $\SBim_R(W)$ to the category $R-\operatorname{mod}$ of graded $R$-modules, are additive, so we can define their extensions to the corresponding bounded homotopy categories, which we denote in the same way:
\[ \HH^i: \Ho(\SBim_R(W)) \to \Ho(R-\operatorname{mod}).
\]

Recall that for a braid $\beta$  we have a corresponding Rouquier complex $F_\beta \in\Ho(\SBim_R(W))$. Khovanov-Rozansky homology of the link $\bar{\beta}$ given by the closure of $\beta$ is defined as $\operatorname{HHH}^i({\beta}) := \operatorname{HH}^i(F_\beta)$. Note that it has three gradings: one coming from the grading of $\operatorname{HHH}^{\bullet}$, one coming from the cohomological degree in $\Ho(\SBim_R(W))$, and one coming from the internal grading on the bimodules.
\section{Generalities on completed categories}
\label{sec:gener-compl-categ} Conventions in this section follow \cite{by}.

\subsection{Frobenius modules.} Fix a finite field $\Fq$ of cardinality $q = p^r$ for a prime $p$, which is always assumed to be very good for the considered algebraic group $G$, and another prime $\ell \neq p$. Fix an isomorphism $\mathbb{C} \simeq \Qlbar$, and let $|\cdot|$ be the corresponding archimedean norm on $\Qlbar$. Fix the square root of $q$ in $\Qlbar$.

Let $\Fr\in \operatorname{Gal}(\Fqbar/\Fq)$ stand for the geometric Frobenius morphism. By $\Fr$-module we mean a $\Qlbar$-vector space $M$ equipped with an automorphism $\Fr_M$. A $\Fr$-module is called locally-finite if it is a union of finite-dimensional $\Fr$-modules. For an arbitrary $\Fr$-module $M$, let $M^f$ be its locally-finite part, that is a union of all its finite-dimensional $\Fr$-submodules.

The weights of a locally-finite Frobenius module $M$ are numbers $2\log |\lambda|/\log q$, where $\lambda$ stands for the generalized eigenvalue of the $\Fr$-action. Since we fixed a square root of $q$, the half of the Tate twist is defined, which we denote by $(1)$ (note that this differs from the usual notations and, in particular, notations from \cite{by}, where $(1)$ is the notation for the full Tate twist). It shifts weights by $-1$. Write $\tw{1} = [1](1)$.

If $M$ is a $\Fr$-module on which $\Fr$ acts semisimply with integral weights, let \[\underline{M} = \bigoplus_{i\in \mathbb{Z}}{M_i}\] be a graded $\Qlbar$-vector space, where $M_i$ is the $i$-weight subspace of $M$.

For a $\Qlbar$-algebra $A$ with $\Fr$-action, let $(A, \Fr)-\operatorname{mod}$ be the category of $A$-modules with a compatible $\Fr$-action. If $\Fr$ acts on both $A$ and $M$ semi-simply with integral weights, $\underline{M}$ becomes a graded $\underline{A}$-module.

\subsection{Categories of sheaves.} Let $Y$ be a scheme of finite type, and let $H$ be an algebraic group, both defined over $\Fq$. We refer the reader to the \cite{laszlo2008six1}, \cite{laszlo2008six2} for the formalism of the $\ell$-adic derived categories on stacks. One may also work with an $\ell$-adic analogue of the construction of \cite{bernstein2006equivariant}. We denote by $D^b(Y/H) = D_H^b(Y)$ the bounded derived category of \'etale $\Qlbar$-sheaves on the quotient stack $Y/H \times_{\Spec \Fq} \Spec \Fqbar$, and by $D_m^b(Y/H)$ the mixed bounded derived category of \'etale $\Qlbar$-sheaves on the quotient stack $Y/H$. For a stack $X$, write $\mathcal{H}^k(-)$ for the $k$th perverse cohomology functor on $D^b_m(X)$ (with respect to the middle perversity). Let $\omega$ stand for the pullback functor $D_m^b(Y/H) \to D^b(Y/H)$. We will regard the category $D_m^b(Y/H)$ as enriched over $\mathbb{Z}[\Fr]$: for $\mathcal{F}, \mathcal{G} \in D_m^b(Y/H)$, let
\[\Hom(\mathcal{F},\mathcal{G}) := \Hom(\omega\mathcal{F},\omega\mathcal{G}),\] \[\operatorname{Ext}^i(\mathcal{F},\mathcal{G}) := \operatorname{Ext}^i(\omega\mathcal{F},\omega\mathcal{G})\] considered as $\Fr$-modules.

\subsection{Generalities on completions} We briefly recall the formalism of pro-objects in filtered triangulated categories, see Appendix A to \cite{by}. Let $\mathcal{D}$ be any category. By $\pro(\mathcal{D})$ we denote its category of pro-objects. Namely, objects of $\pro(\mathcal{D})$ are sequences \[X_0 \leftarrow X_1 \leftarrow X_2 \leftarrow \dots, X_i \in \mathcal{D},\] denoted by $\prolim X_{\bullet}$, and
\[ \Hom_{\pro(\mathcal{D})}(\prolim X_{\bullet}, \prolim Y_{\bullet}) = \invlim{n}\dirlim{m}\Hom_{\mathcal{D}}(X_m, Y_n).
\]

Let $T$ be a split torus of rank $r$, and let $\pi:X \to Y$ be a $T$-torsor. Following \cite{by}, consider $D' := D'_m(X) \subset D^b_m(X)$ -- unipotently monodromic subcategory.

\begin{bdefinition}
  \label{sec:gener-compl} The completed monodromic category $\hat{D}(X) \subset \pro(D')$ is a full subcategory of sequences $\prolim \mathcal{F}_{\bullet}$ such that
  \begin{enumerate}[label=(\arabic*)]
  \item
    \label{item:1} $\mathcal{F}$ is uniformly bounded in degrees: $\mathcal{F}_{\bullet} \simeq \mathcal{F}'_{\bullet}$ (isomorphism in $\operatorname{pro}(D')$) with $\mathcal{F}'_{\bullet}$ such that there is $N > 0$ for which, for all $n$, $\mathcal{F}'_n$ has no perverse cohomology outside the interval $[-N, N]$.
  \item
    \label{item:2} $\mathcal{F}$ is uniformly bounded above in weights: $\mathcal{F}_{\bullet} \simeq \mathcal{F}'_{\bullet}$ with $\mathcal{F}'_{\bullet}$ such that there is $N \in \mathbb{Z}$ for which, for all $n$, $\mathcal{F}'_n$ is of weight $\leq N$.
  \item
    \label{item:3} $\pi_!\prolim \mathcal{F}_{\bullet}$ (equivalently, $\pi_*\prolim \mathcal{F}_{\bullet}$, see Lemma \ref{*!} below) lies in the essential image of $D^b_m(Y)$ in $\pro(D^b_m(Y))$.
  \end{enumerate}
\end{bdefinition}

It was shown in loc. cit. that $\hat{D}(X)$ is a triangulated category.

We will also consider the stratified situation, as in A.6 of \cite{by}. Assume that we are in the situation of Assumption S of loc. cit. Namely, $Y$ is assumed to be stratified with affine strata $Y_{\alpha}, \alpha \in \mathcal{S}$, $X_\alpha = \pi^{-1}(X_{\alpha})$ are trivial $T$-torsors over $Y_\alpha$, $\operatorname{H}^*(Y_{\alpha}\otimes_{\Fq}\Fqbar) \simeq \Qlbar$. Write $D_{\leq \alpha}(Y)$ (resp. $D_{< \alpha}(Y)$) for $D^b_m(\bar{Y}_{\alpha})$ (resp. $D^b_m(\bar{Y}_{\alpha} \backslash Y_{\alpha})$). We fix a full triangulated subcategory $\mathcal{D} \subset D_m^b(Y)$. For each $\alpha,$ let $\mathcal{D}_{\leq \alpha}$ (resp. $\mathcal{D}_{< \alpha}$) stand for the category $\mathcal{D} \cap D_{\leq \alpha}(Y)$ (resp. $\mathcal{D} \cap D_{< \alpha}(Y)$, and write $\mathcal{D}_{\alpha} = \mathcal{D}_{\leq \alpha}/\mathcal{D}_{< \alpha}$. Assume that the heart of $\omega\mathcal{D}_{\alpha}$ (with respect to the perverse t-structure) has a unique simple object, given by a rank one perverse local system $\mathcal{L}_{\alpha} \in \mathcal{D}_{\alpha}$. Let $d_{\alpha} = \dim Y_{\alpha}$. Let $\tilde{j}_{\alpha}: X_{\alpha} \to X$ be the corresponding embedding.

The main example we use is $X = G/U, Y = G/B$ with Bruhat stratification, and with $\mathcal{D}$ being the derived category of $U$-equivariant mixed complexes on $G/B$.

Let $\M$ be the stratified unipotently monodromic category attached to the above data, that is the full triangulated subcategory of $D^b_m(X)$ generated by objects $\pi^*\F, \F \in \mathcal{D}$, and let $\cM \subset \pro(\M)$ be its monodromic completion, given by replacing $D'$ in Definition \ref{sec:gener-compl} by $\M$.
\subsection{Standard and costandard pro-objects.}
\label{sec:stand-cost-pro} Define $V_T$ to be the $\Qlbar$-Tate module of $T$. This is a Frobenius module of weight $-2$. Write $S = \operatorname{Sym}V_T$. Let $\mathcal{L}_n$ be the unipotent local system on $T$ corresponding to the representation $S/(V_T)^n$ of $\pi_1^{\acute{e} t}(T)$ (that factors through its $\ell$-adic tame quotient), $\LL = \invlim{n} \mathcal{L}_n$ be the free pro-unipotent local system. Since, by assumption, $X_{\alpha}$ is a trivial $T$-torsor over $Y_{\alpha}$, we may consider the corresponding local system on $X_{\alpha}$, which we denote in the same way.

Recall that $r$ stands for the rank of $T$.

Standard and costandard pro-objects in $\cM$ are defined as \[\hat{\Delta}_\alpha = \tilde{j}_{\alpha!} \LL[r + d_{\alpha}](2r + d_{\alpha}), \hat{\nabla}_\alpha = \tilde{j}_{\alpha*} \LL[r + d_{\alpha}](2r + d_{\alpha}), \] respectively.

In case $Y_{\alpha} = \operatorname{pt}$ (in all situations we encounter there is a unique stratum of this form), we adopt the notation $\dd := \LL[r](2r)$ for the unique (up to a non-unique isomorphism) standard (and costandard) pro-object on $X_{\alpha} \simeq T$.

Let $a: T \times X \to X$ be the action morphism. We have the following
\begin{lemma}[\cite{by}, Lemma A.3.6]
  \label{unit} There is a natural isomorphism $a_!(\LL\boxtimes\F) \simeq \F[-2r](-2r)$, $\F \in \cM$.
\end{lemma} The following will be useful to us on multiple occasions.

\begin{lemma}
  \label{*!} We have the following isomorphism of functors:
  \[ \pi_* \simeq \pi_![r] : \cM \to D^b_m(Y).
  \]
\end{lemma}
\begin{proof} Consider the diagram
  \[
\begin{tikzcd} T\times X \arrow[d, "a"'] \arrow[r, "\pi_2"] & X \arrow[d, "\pi"] \\ X \arrow[r, "\pi"'] & Y
\end{tikzcd}
\] Here $a$ is the action morphism, $\pi_2$ is the projection to the second factor. Note that the diagram is a morphism of $T$-torsors $T\times X \to X$. We have a natural isomorphism of functors $\pi_*a_! \simeq \pi_!\pi_{2*}$ (see \cite{by}, proof of Lemma A.3.4 (3)). Applying it to the sheaf $\LL \boxtimes \F,$ for $\F \in \cM$, we get a natural in $\F$ isomorphism
\[ \pi_*\F[-2r](-2r) \simeq \pi_*a_!\LL \boxtimes \F \simeq \pi_!\pi_{2*} \LL \boxtimes \F \simeq \pi_! \F [-r](-2r).
\]

Here the last isomorphism follows from the isomorphism
\[ \operatorname{H}^{\bullet}(\LL) \simeq \Qlbar [-r](-2r),
\] where $\Qlbar$ stands for the trivial $\Fr$-module on a point.
\end{proof}
\subsection{Verdier duality.} In this subsection we describe the Verdier duality formalism we will use.

We have an action of $D^b_m(T)$ on $\cM$ defined by
\[ \mathcal{F}\star\mathcal{A} = a_!(\mathcal{F}\boxtimes\mathcal{A})[r],
\] for $\mathcal{F} \in D^b_m(T)$, and $\mathcal{A} \in \cM$.

\begin{lemma}
  \label{sec:verdier-duality-2} Assume that $\F$ is a unipotently monodromic complex on $T$ (considered as a $T$-torsor $T \to \operatorname{Spec}\Fq$). Then, for any $\mathcal{A} \in \cM$, $\F\star\mathcal{A}$ is in the essential image of $\M$ in $\cM$.
\end{lemma}
\begin{proof} It is enough to check the statement of the Lemma for $\F = \cc_T$. For such $\F$ it follows from the proper base change and property \ref{item:3} of Definition \ref{sec:gener-compl}.
\end{proof} Let $\D$ stand for the Verdier duality functor.

\begin{lemma}
  \label{sec:verdier-duality} There is a natural isomorphism
\[ \D(\mathcal{F}\star\mathcal{A}) \simeq \D\F\star\D\mathcal{A}[-r].
\] for $\mathcal{F} \in D^b_m(T)$, and $\mathcal{A} \in \M$.
\end{lemma}

\begin{proof} By the standard properties of the Verdier duality, we have
  \[
\D(\mathcal{F}\star\mathcal{A}) = \D(a_!(\mathcal{F}\boxtimes\mathcal{A})[r]) \simeq a_*(\D\mathcal{F}\boxtimes \D\mathcal{A})[-r].
  \]
  On the other hand, by Lemma \ref{*!} applied to the torsor $a: T \times X \to X$, if $\mathcal{A}$ is in the image of $\M$, we have
  \[ a_*(\D\mathcal{F}\boxtimes \D\mathcal{A})[-r] \simeq a_!(\D\mathcal{F}\boxtimes\D\mathcal{A}) = \D\F\star\D\mathcal{A}[-r].
  \]
\end{proof}

Let $\delta_n = \mathcal{L}_n[r](2r)$, so that $\dd = \text{``$\invlim{n}$''}\delta_n$.

By Lemma \ref{unit}, we have
\[ \mathcal{F} \simeq \invlim{n}\left({\delta_n}\star\mathcal{F}\right)
\] for any $\mathcal{F} \in \M$. By Lemma \ref{sec:verdier-duality} we then get
\[ \D(\mathcal{F}) \simeq \invlim{n}\left({\delta_n}\star\D\mathcal{F}\right) \simeq \invlim{n}\D(\D({\delta_n})\star\mathcal{F})[r],
\] for any $\mathcal{F} \in \M$.

This motivates the following definition of the duality functor on $\cM$.
\begin{bdefinition}
\label{sec:verdier-duality-1}
  Let $\D:\cM \to \pro(\M)$ be the functor
  $$
  \D(\prolim\mathcal{F}_{\bullet}) = \text{``$\invlim{n}$''}\D(\D({\delta_n})\star\mathcal{F})[r].
  $$
for $\mathcal{F} = \prolim\mathcal{F}_{\bullet} \in \cM$.
\end{bdefinition}

Note that by Lemma \ref{sec:verdier-duality-2}, $\D(\delta_n)\star \F$ is in the essential image of $\M$, so $\D\F$ for $\F\in\cM$ can indeed be considered an object of $\pro(\M)$.
 \begin{lemma} For any $\F$ in $\cM$, the object $\D\F$ is also in $\cM$.
 \end{lemma}
 \begin{proof} Note that if $\mathcal{F}$ is uniformly bounded in degrees, so is $\D\mathcal{F}$, since convolution with a perverse sheaf on $T$ has a perverse cohomological amplitude bounded by $r$. Thus, the property \ref{item:1} of Definition \ref{sec:gener-compl} is satisfied.

   We also have
   \begin{multline*} \pi_!\text{``$\invlim{n}$''}\D(\D({\delta_n})\star\mathcal{F})[r] \dot{=}\text{``$\invlim{n}$''}\D\pi_*(\D({\delta_n})\star\mathcal{F}) \dot{=} \\ \dot{=} \text{``$\invlim{n}$''}\D(\operatorname{H}^{\bullet}(\D{\delta_n}) \otimes_{\Qlbar} \pi_*\F)) \dot{=} \D(\pi_*\F),
   \end{multline*} where $\dot{=}$ denotes an isomorphism up to shifts and twists independent of $n$. Here the first isomorphism is just an isomorphism $\D\pi_* \simeq \pi_!\D$, the second one is the projection formula, and the third is an observation that $\invlim{n}\left(\D\operatorname{H}^{\bullet}(\D\delta_n)\right)\simeq\invlim{n}\operatorname{H}^{\bullet}_c(\delta_n)\simeq \Qlbar[-r]$. This proves that $\D$ preserves property \ref{item:3} of Definition \ref{sec:gener-compl}.

   It remains to show that $\D$ preserves property \ref{item:2} of Definition \ref{sec:gener-compl}. It is enough to show that if $\F = \prolim \F_{\bullet}$ is uniformly bounded above in weights, then $(\D\delta_n\star \F)$ is uniformly bounded below in weights.

   Indeed, under our assumptions, the category $\cM$ is generated as a triangulated category by the shifts and twists of costandard objects $\NN_{\alpha}$, see Lemma A.6.1 of \cite{by}. By Lemma \ref{*!} applied to $a:T\times X\to X$, we have
     \[ \D\delta_n \star \NN_{\alpha} \dot{=} \tilde{j}_{\alpha*}\D\mathcal{L}_n.
     \]
     
     Weights of $\D\mathcal{L}_n$ are in the interval $[-2r, 2(n-r)]$ and $\tilde{j}_{\alpha_*}$ does not decrease weights, so that $\D\delta_n\star \NN_{\alpha}$ is uniformly bounded below in weights, as needed.

 \end{proof}
 
\begin{lemma} The functor $\D: \cM \to \cM$ satisfies the following natural properties of the duality functor:
  \begin{enumerate}[label=(\arabic*)]
  \item $\D\D \simeq \operatorname{Id}$.
  \item $\D[1] \simeq [-1]\D$.
  \item $\Hom(\mathcal{A}, \mathcal{B}) \simeq \Hom(\D \mathcal{B}, \D \mathcal{A})$ for $\mathcal{A}, \mathcal{B} \in \hat{D}(X)$.
  \end{enumerate}
\end{lemma}
\begin{proof} Second claim is obvious and the first one follows from the third and Yoneda lemma. For the third claim, note that the functor $({\delta_n}\star -)$ admits both left and right adjoint on the monodromic category, namely $(\D{\delta_n}[-r]\star -)$, so we have
  \begin{multline*} \Hom(\D \mathcal{B}, \D \mathcal{A}) = \invlim{n}\dirlim{m}\Hom(\D(\D{\delta_m}\star\mathcal{B}), \D (\D{\delta_n}\star\mathcal{A})) \simeq \\ \simeq \invlim{n}\dirlim{m}\Hom(\D{\delta_n}\star\mathcal{A}, \D{\delta_m}\star\mathcal{B}) \simeq \invlim{n}\dirlim{m}\Hom({\delta_m}\star\mathcal{A}, {\delta_n}\star\mathcal{B}) = \Hom(\mathcal{A}, \mathcal{B}).
  \end{multline*}
\end{proof}

The following records how $\D$ acts on standard and costandard pro-objects.
\begin{lemma} $\D\DD_{\alpha} \simeq \NN_{\alpha}[r], \D\NN_{\alpha} \simeq \DD_{\alpha}[r]$.
\end{lemma}
\begin{proof} We have
  \begin{multline*} \D\delta_m\star \DD_{\alpha} = \invlim{n}\D\delta_m \star \tilde{j}_{\alpha !}\mathcal{L}_n[r+d_\alpha](2r+d_\alpha) \simeq \\ \simeq \invlim{n}\tilde{j}_{\alpha !}\left(\D\delta_m\star \mathcal{L}_n\right)[r+d_\alpha](2r+d_\alpha)\simeq \tilde{j}_{\alpha !}\D\mathcal{L}_m[r+d_\alpha](2r+d_\alpha) \simeq \\ \simeq \D \tilde{j}_{\alpha *}\mathcal{L}_m[r+d_\alpha](2r+d_\alpha).
  \end{multline*} and so
  \[ \D\DD_\alpha = \text{``$\invlim{n}$''}\D(\D(\delta_n)\star\DD_{\alpha})[r] \simeq \text{``$\invlim{n}$''}\tilde{j}_{\alpha *}\mathcal{L}_n[2r+d_\alpha](d_\alpha) = \NN_{\alpha}[r].
  \] The second claim follows, since $\D\D \simeq \operatorname{Id}$.
\end{proof}
\subsection{t-structures.}
\label{t-structures} Recall from \cite{by} that $\cM$ admits a perverse t-structure \[\left(\cM^{\leq 0}, \cM^{\geq 0}\right),\] with respect to which the natural functor $\M\to\cM$ is perverse t-exact. We recall its definition.

Let $\cM_{\leq \alpha}$ be the full subcategory of sheaves supported on strata lying in the closure of $X_{\alpha}$. Let $\cM_{< \alpha}$ be the full subcategory of sheaves supported on the union of strata $\beta \neq \alpha$ lying in the closure of $X_{\alpha}$. Let $\cM_{\{\alpha\}}$ be the Verdier quotient category $\cM_{\leq \alpha}/\cM_{< \alpha}$. We have functors $i^*_\alpha, i^!_\alpha:\cM \to \cM_{\{\alpha\}}$ and the categories $\cM_{\{\alpha\}}$ were proved in loc. cit. to be equivalent to the bounded derived category of $\mathcal{D}_m := D^b(S,\Fr-\operatorname{mod})$ of finitely generated graded $S$-modules with the compatible Frobenius action. Let $(\mathcal{D}_m^{\leq 0}, \mathcal{D}_m^{\geq 0})$ stand for the standard t-structure on the latter category. Now define the t-structure on $\cM$ by gluing as follows:
\[ \cM^{\leq 0} = \{\F \in \cM,\forall \alpha \in \mathcal{S}\ i_\alpha^*\F \in \mathcal{D}_m^{\leq 0}\},
  \]
   \[ \cM^{\geq 0} = \{\F \in \cM,\forall \alpha \in \mathcal{S}\ i_\alpha^!\F \in \mathcal{D}_m^{\geq 0}\}.
 \]

Let $\Perv$ be its heart. We have $\hat{\Delta}_\alpha, \hat{\nabla}_\alpha \in \Perv$ for all $\alpha \in \mathcal{S}$, and $\omega\cM^{\leq 0}$ is generated under extensions by $\omega\DD_{\alpha}[n], n \geq 0$. We denote by $\tau^{\geq k}, \tau^{\leq k}, \mathcal{H}^k(-)$ the truncations and perverse cohomology functors with respect to this t-structure, respectively.

This t-structure is not self-dual with respect to $\D$, cf. Remark \ref{sec:t-structure-duality-2}. Consider the shifted duality functor, $\D' = \D[-r]$, chosen so that $\D'\DD_{\alpha} \simeq \NN_{\alpha}$, and consider the dual t-structure
\[ \left(\cM'^{\leq 0} = \D'\cM^{\geq 0}, \cM'^{\geq 0} = \D'\cM^{\leq 0}\right).
\] Let $\Perv'$ be its heart. We have $\DD_{\alpha}, \NN_{\alpha} \in \Perv'$ for all $\alpha \in \mathcal{S}$, and $\cM'^{\geq 0}$ is generated under extensions by $\NN_{\alpha}[n], n \leq 0$.
\section{Monodromic Hecke category}
\label{sec:monodr-hecke-categ} Let $G$ be a split reductive group over $\Fq$. Fix $B \subset G$ a split Borel subgroup, $T \subset B$ split maximal torus. Let $U \subset B$ be the unipotent radical. Let $\M \subset D_m^b(U\backslash G/U)$ be the mixed derived category of complexes that are unipotently monodromic with respect to the right $T$-action.

So, from now on $X = G/U, Y = G/B$, and we are considering the stratification of $Y$ by left $U$-orbits. In this setting, the set $\mathcal{S}$ of strata is identified with the Weyl group $W$. The stratum labeled by an element $w \in W$ has dimension $d_w = l(w)$.

\subsection{Convolution.}
\label{sec:convolution}
Categories $D^b_m(U\backslash G/U), \M, \cM, \omega\M, \omega\cM$ are equipped with monoidal structure via the $!$-convolution, which we denote by $\star$. We recall their definitions. Consider the diagram
\[
\begin{tikzcd} & G\times^U G/U \arrow[r, "q"] \arrow[ld, "\pi_1"'] \arrow[rd, "\pi_2"] & G/U \\ G/U & & U\backslash G/U
\end{tikzcd}
\] Here $\pi_1, \pi_2$ are projections and $q$ is the action map. Convolution operation is defined as
\begin{equation}
\label{eq:3} \F \star \mathcal{G} = q_!(\F \boxtimes \mathcal{G})[r].
\end{equation}

By Lemma 4.3.1 of \cite{by}, the convolution is well-defined on the completed category $\cM$.

We have the following
\begin{lemma} For $\F, \mathcal{G} \in \cM$
\[ \D(\mathcal{F}\star\mathcal{G}) \simeq \D\F\star\D\mathcal{G}[-r].
\]
\end{lemma}
\begin{proof} To prove this fact we will need a representation of $\dd$ as a projective system of objects $\prolim\varepsilon_n$ that are central, that is such that we have a canonical isomorphism $\F\star\varepsilon_n \simeq \varepsilon_n\star\F$ for any $\F \in \cM$. We will thus use the following fact from Corollary \ref{centralunit} below.  Let $\varepsilon_n = p^*\hc_!(\mathcal{E}_n)$ (see Corollary \ref{centralunit} for notations). We have $\dd = \invlim{n}\varepsilon_n$ and $\varepsilon_n, \D\varepsilon_n$ are central in $\cM$, see Section \ref{sec:centr-struct-harish-1}. For the purposes of the Lemma, the only property of $\varepsilon_n$ we need is the commutativity isomorphism as above.
  
  The isomorphism above follows from Lemma \ref{*!} for $\mathcal{F}, \mathcal{G} \in \M$, since $q$ factors as the composition of a $T$-torsor projection $G \times^U G/U \to G \times^B G/U$ and a proper map $G \times^B G/U \to G/U$. On the other hand, from the proof of Lemma 4.3.1 in \cite{by} one deduces that, for $\mathcal{A} = \text{``$\invlim{m}$''}\mathcal{A}_m, \mathcal{B} = \text{``$\invlim{n}$''}\mathcal{B}_n$, we have $\mathcal{A} \star \mathcal{B}_n \in \M$ and $\mathcal{A}\star\mathcal{B} \simeq \text{``$\invlim{n}$''}\left(\mathcal{A} \star \mathcal{B}_n\right)$. By definition,
  \[ \D\F \star \D \mathcal{G} = \text{``$\invlim{m}$''}\D(\D\varepsilon_m\star\F)\star \text{``$\invlim{n}$''}\D(\D\varepsilon_n\star\mathcal{G})
  \] and we also have,
  \[
    \invlim{m}\left(\D(\D\varepsilon_m\star\F)\star\D(\D\varepsilon_n\star\mathcal{G})\right) \simeq \invlim{m}\D((\D\varepsilon_m\star\F)\star(\D\varepsilon_n\star\mathcal{G}))[r],
  \]
  since $\D\varepsilon_m\star\F, \D\varepsilon_n\star\mathcal{G} \in \M$. Using the central structure to rearrange the factors, we get that the last expression is isomorphic to 
\[
  \invlim{m}\D(\D\varepsilon_m\star(\D\varepsilon_n\star\F\star\mathcal{G}))[r] \simeq \invlim{m}\varepsilon_m\star\D(\D\varepsilon_n\star\F\star\mathcal{G}) \simeq \D(\D\varepsilon_n\star(\F\star\mathcal{G})).
\]
  Passing to the limit in $n$ we get the result.
\end{proof}
\subsection{Convolution of standard and costandard pro-objects.} Standard and costandard pro-objects satisfy the following properties with respect to convolution:
\begin{proposition}[\cite{by}, \cite{br}]~
 \label{braid_relations}
  \begin{enumerate}[label=(\arabic*)]
  \item\label{item:18} $\dd$ is the unit of the monoidal structure $\star$.
  \item\label{item:19} $\DD_v \star \DD_w \simeq \DD_{vw}, \NN_v \star \NN_w \simeq \NN_{vw}$ if $l(vw) = l(v) + l(w)$.
  \item\label{item:20} $\DD_v \star \NN_{v^{-1}} \simeq \dd.$
  \item \label{item:11} $\Hom(\DD_v, \NN_w[i]) = 0,$ unless $v = w$ and $i = 0$, and $\Hom(\DD_v, \NN_v) \simeq \hat{S}$, where $\hat{S}$ is the completion of $S$ with respect to the ideal $V_TS$.
  \end{enumerate}
\end{proposition}
\begin{proof} \ref{item:18} and \ref{item:19} are Lemma 4.3.3 of \cite{by}, \ref{item:20} is Lemma 7.7 of \cite{br}, \ref{item:11} is a standard computation using the adjunction $(\tilde{j}_{w}^*, \tilde{j}_{w*})$.
\end{proof}
\subsection{Pro-unipotent tilting sheaves.}

Recall that in \cite{by} a category of free-monodromic tilting objects in $\cM$ was defined, coming with a collection of indecomposable objects $\T_w \in \Perv$ indexed by $w \in W$.

Let \[\mathbb{V}: \cM \to (\operatorname{End}(\T_{w_0})^f,\Fr)-\operatorname{mod}\] be the functor $\Hom(\T_{w_0}, -)^f$.

\begin{proposition}[\cite{by}]~
 \label{tiltprop}
  \begin{enumerate}[label=(\arabic*)]
  \item \label{tiltprop1} For all $w \in W$, $\T_w$ admits a filtration by objects of the form $\DD_v(m), m \in \mathbb{Z}$, referred to as the standard filtration, and also a filtration by objects of the form $\NN_v(m), m \in \mathbb{Z}$, referred to as the costandard filtration.
  \item\label{item:21} $\Hom(\T_w, \T_v[i]) = 0, v, w \in W, i > 0.$
  \item\label{item:23} $\Hom(\T_v, \T_w)^f \simeq \Hom(\mathbb{V}(\T_v),\mathbb{V}(\T_w)), v, w \in W,$ as $\Fr$-modules.
  \item\label{item:24} $\operatorname{End}(\T_{w_0})^f \simeq S\otimes_{S^W} S$ as $\Fr$-algebras.
  \end{enumerate}
\end{proposition}
\begin{proof} \ref{tiltprop1} follows directly from the definition, given in \cite{by}, A.7. \ref{item:21} follows from \ref{tiltprop1} and Proposition \ref{braid_relations}. \ref{item:23} is Proposition 4.5.7 together with Lemma 4.6.3 of loc. cit. \ref{item:24} is Proposition 4.7.3 (2) of loc. cit.
\end{proof} 
\begin{bdefinition}
  \label{sec:pro-unip-tilt-2} Define the shifted t-structure on $\cM$ as
  \[ \left(\cM_{w_0}^{\geq 0} = \DD_{w_0} \star \cM'^{\geq 0}, \cM_{w_0}^{\leq 0} = \DD_{w_0} \star \cM'^{\leq 0}\right)
\]
We denote by $\tau_{w_0}^{\geq k}, \tau_{w_0}^{\leq k}, \mathcal{H}^k_{w_0}(-)$ the truncations and perverse co\-ho\-mo\-lo\-gy functors with respect to this t-structure, respectively, and let $\Perv_{w_0}$ be its heart.
\end{bdefinition}
This is the pullback of the t-structure $\left(\cM'^{\geq 0}, \cM'^{\leq 0}\right)$ defined in Section \ref{t-structures} along the autoequivalence $(\NN_{w_0}\star -)$ of $\cM$: the complex $\F$ is the heart of the t-structure $\left(\cM^{\geq 0}_{w_0}, \cM^{\leq 0}_{w_0}\right)$ if and only if $\NN_{w_0}\star \F$ is in the heart of the t-structure $\left(\cM'^{\geq 0}, \cM'^{\leq 0}\right)$.

Note that by Proposition \ref{braid_relations} and Proposition \ref{tiltprop} \ref{tiltprop1}, $\T_w \in \Perv_{w_0}$: since $\T_w$ has a filtration by objects of the form $\DD_v, v \in W$, and $\NN_{w_0}\star \T_w$ has a filtration by objects of the form $\NN_{w_0}\star \DD_v \simeq \NN_{w_0v}$.

We will need the following
\begin{proposition}
  \label{sec:pro-unip-tilt} For any $\F \in \cM, v \in W, k\in \mathbb{Z}$, \[\Hom(\F, \T_v[k]) \simeq \Hom(\mathcal{H}_{w_0}^{-k}(\F), \T_v).\]
\end{proposition}
\begin{proof} It is enough to show that for any $\F \in \Perv_{w_0}$, $\Hom(\F, \T_v[k]) = 0$ unless $k = 0.$ Since $\F, \T_v \in \Perv_{w_0}, \Hom(\F, \T_v[k]) = 0$ for $k < 0$. It remains to consider the case $k > 0$. The category $\omega\cM'^{\geq 0}$ is generated under extensions by the objects $\omega\NN_w[n], w \in W, n \leq 0$, so the category $\omega\cM^{\geq 0}_{w_0}$ is generated under extensions by the objects $\omega\DD_{w_0}\star \omega\NN_{w}[n] \simeq \omega\DD_{w_0w}[n], w \in W, n \leq 0.$ On the other hand, by Proposition \ref{tiltprop} \ref{tiltprop1} $\omega\T_v$ admits a filtration by sheaves of the form $\omega\NN_v$. We have
  \[ \Hom(\DD_{w_0} \star \NN_{w}[n], \NN_{v}[k]) \simeq \Hom(\DD_{w_0w}, \NN_{v}[-n+k]) = 0,
  \] by Proposition \ref{braid_relations} \ref{item:11} using that $-n+k > 0$, which implies the result.
\end{proof}

We will also use the following
\begin{lemma}
  \label{sec:pro-unip-tilt-1} There is an isomorphism $\D'\T_v \simeq \T_v.$
\end{lemma}
\begin{proof} By Lemma 5.2.2 of \cite{by}, any indecomposable free-monodromic tilting extension of $\LL[r+d_{v}](2r+d_{v})$ from the stratum labeled by $v$ is isomorphic to $\T_v$. Since $\D'$ exchanges free-monodromic standard and costandard objects, we get that $\D'\T_v$ is an indecomposable free-monodromic tilting extension of $\LL[r+d_{v}](2r+d_{v})$, and the Lemma follows.
\end{proof}
\subsection{Comparison with Soergel bimodules.} Let $\Tilt$ be the additive category generated by the twists of the free-monodromic tilting sheaves $\T_v$. By Proposition 4.3.4 of \cite{by}, $\Tilt$ is closed under the monoidal structure on $\cM$.

The connection between the setting of Soergel bimodules and monodromic Hecke categories can be summarized in the following diagram:
\[
\begin{tikzcd} \SBim_R(W) \arrow[d, shift right, "\iota_{\mathbb{S}}"'] \arrow[r, "\Lambda"] & \Tilt \arrow[d, "\iota_T"] \arrow[r, "\omega"] & \omega\Tilt \arrow[d, "\omega\iota_T"] \\ \Ho(\SBim_R(W)) \arrow[r, "\Ho(\Lambda)"'] & \Ho(\Tilt) \arrow[r, "\omega"] \arrow[d, "h"'] & \Ho(\omega Tilt) \arrow[d, "\wr", "\omega h"'] \\ & \cM \arrow[r, "\omega"'] & \omega \cM
\end{tikzcd}
\]

Here we consider the category $\SBim_R(W)$ of Soergel bimodules associated to the representation $\mathfrak{h}_0 = V_T^{\vee}$ of $W$, as in Section \ref{sec:khov-rozansky-homol}, $\iota_{\mathbb{S}}, \iota_T$ stand for the embedding of the additive categories to their homotopy categories, $h$ is the functor constructed in \cite{by}, Appendix B. The functor $\Lambda$ sends $B_w(k)$ to $\T_w(-k)$, and the action of $R \otimes_{\Qlbar} R$ is identified with the action of $S \otimes_{\Qlbar} S$ by the logarithms of the left and right monodromy automorphisms. The functor inverse to $\Lambda$ on its image is $\T \mapsto \underline{\mathbb{V}(\T)^{\Frinv}}$, with the above identification of $S$ and $R$ (recall that for a Frobenius module $M$, $M^{\Frinv}$ denotes the module with inverse Frobenius action). The grading on $\SBim_R(W)$ corresponds to the negative of the Frobenius weights on $\Tilt$.

\begin{proposition}
\label{sec:comp-with-soerg-4} The functor $h\Ho\Lambda$ is monoidal and satisfies the following properties.
  \begin{enumerate}[label=(\arabic*)]
\item \label{item:22} We have isomorphism
  \[ h\Ho\Lambda(\nabla_w) \simeq \NN_w, h\Ho\Lambda(\Delta_w) \simeq \DD_w.
  \]
\item\label{item:12} For $M^{\bullet}, N^{\bullet} \in \Ho(\SBim_R(W)),$
    \[ \Hom(M, N) \simeq \underline{\Hom(h\Ho(\Lambda)(M), h\Ho(\Lambda)(N))^{f, \Frinv}}.
    \]

\item \label{item:13}The functor $h\Ho(\Lambda)$ is t-exact with respect to the t-structure \[( \prescript{p}{}\SBim_R(W)^{\leq 0}, \prescript{p}{}\SBim_R(W)^{\geq 0})\] on $\Ho(\SBim)$ and the t-structure $(\cM^{\leq 0}, \cM^{\geq 0})$ on $\cM$.

\item \label{item:14}The functor $h\Ho(\Lambda)$ intertwines the duality functor $^\vee$ with $\D'$:
  \[ h\Ho(\Lambda)(M^\vee) \simeq \D'h\Ho(\Lambda)(M).
  \]
  \end{enumerate}

\end{proposition}
\begin{proof} The fact that $h\Ho(\Lambda)$ is monoidal follows from \cite{by} Propositions 4.3.4, 4.6.4 (2) and Corollary 5.2.3. It follows that it is enough to check \ref{item:22} on simple reflections, which is done in the Appendix C of loc. cit.

  The identity \ref{item:12} of $\Hom$-spaces is Lemma B.5.1 of \cite{by}.

Since the category $\SBim_R(W)^{\leq 0}$ is generated by the grading shifts of the objects $\Delta_w[n], n \geq 0$, and the category $\cM^{\leq 0}$ is generated by the twists of the objects $\DD_w[n], n \geq 0$, and we have $h\Ho\Lambda(\Delta_w) \simeq \DD_w$, we get that $h\Ho\Lambda$ is t-exact from the right.

For any object $X \in \Ho(\SBim_R(W))$ consider the triangle
\[ \tau_{\leq 0}X \to X \to \tau_{\geq 1}X \to \tau_{\leq 0}X[1].
\] Using \ref{item:12}, we get that
\[ \underline{\Hom(\DD_w[n],\hHoL(\tau_{\geq 1}X))^{f, \Frinv}} \simeq \Hom(\DD_w[n],\tau_{\geq 1}X) = 0,
\] for $n\geq 0$, so that $\hHoL(\tau_{\geq 1}X)\in \cM^{\geq 1}$. By right t-exactness of $\hHoL$, we also have $\hHoL(\tau_{\leq 0}X)\in \cM^{\leq 0}$, and so
\[\hHoL(\tau_{\geq 1}X) \simeq \tau_{\geq 1}\hHoL(X), \hHoL(\tau_{\leq 0}X) \simeq \tau_{\leq 0}\hHoL(X),\] which proves \ref{item:13}.

The fact that the functor intertwines the duality follows from the corresponding statment for $\Lambda$, see Lemma ~\ref{sec:pro-unip-tilt-1} .
\end{proof}
\begin{corollary}
  \label{sec:comp-with-soerg-5} The functor $hHo(\Lambda)$ is t-exact with respect to the t-structure $\left(\SBim_R(W)_{w_0}^{\geq 0},\SBim_R(W)_{w_0}^{\leq 0}\right)$ on $\Ho(\SBim_R(W))$, see Definition \ref{sec:t-structure-duality-1} and the t-structure $\left(\cM_{w_0}^{\geq 0}, \cM_{w_0}^{\leq 0}\right)$, see Definition \ref{sec:pro-unip-tilt-2}.
\end{corollary}
\begin{bdefinition}
  \label{sec:comp-with-soerg-6}
Fix $w \in W$ and consider the following twisted adjoint action of $T$ on $G/U$: \[t\cdot xU = tx\operatorname{Ad}(w)(t^{-1})U.\] 
Let $\mathcal{Y}_w$ be the quotient stack $(U\backslash G/U)/{T}$ with respect to this action. We will write $\mathcal{Y} = \mathcal{Y}_1,$ with $p:U\backslash G/U \to \mathcal{Y}$ for the corresponding projection.
\end{bdefinition}

Let $p^{\dagger} = p^*[\dim T]$. Thus $p^{\dagger}$ is the t-exact functor forgetting the $T$-action.

We now give a geometric description of the complex $K_{\mathbb{S}} \in \Ho(\SBim)$.
\begin{proposition}
\label{sec:comp-with-soerg} There is an isomorphism
 \[ h\Ho(\Lambda)(K_{\mathbb{S}}) \simeq p^{\dagger}p_{!}\T_{w_0}\tw{2\dim T}.
  \]
\end{proposition}

The proposition follows from the following Lemma \ref{sec:comp-with-soerg-3}, which is a straightforward adaptation of a result of \cite{2geometric} to the mixed setting.

Let $X$ be a scheme of finite type over $\mathbb{F}_q$ equipped with an action $T$. Consider a (unital) commutative dg-algebra $\mathbb{K}_T$ with Frobenius action freely generated by $V_T[1](4)\oplus V_T$ (recall that $\Fr$ acts on $V_T$ by $q^{-1}$ and on $V_T[1](4)$ by $q$).

Let $\mathcal{P}_m'(X) \subset D^b_m(X)$ be the category of mixed monodromic perverse sheaves on $X$. Let $\mathcal{P}_m^{eq}$ be the dg-category of bounded complexes $P^{\bullet}$ of objects in $\mathcal{P}_m'(X)$ equipped with an action of $\mathbb{K}_T$ compatible with the Frobenius action and monodromy: that is, equipped with a homomorphism of $\Fr$-dg-algebras $\mathbb{K}_T \to \Hom(\omega P^{\bullet},\omega P^{\bullet})$, and such that the action of $V_T$ is given by the logarithm of monodromy action. Let $D(P_m^{eq})$ be the corresponding derived category.

\begin{lemma}
  \label{sec:comp-with-soerg-3} There is an equivalence of triangulated categories $real_{eq}:D(P_m^{eq})\to D^b_m(X/T)$ such that the diagrams above commute up to natural isomorphism:
\[
  \begin{tikzcd} D(P^{eq}_m) \arrow[d, "real_{eq}"'] \arrow[r, "\operatorname{For}"] & D^b(P'_m) \arrow[d, "real"] \\ D^b_m(X/T) \arrow[r, "p^*"] & D_m^b(X)
\end{tikzcd}
\]
\[
  \begin{tikzcd} D(P^{eq}_m) \arrow[d, "real_{eq}"'] & D^b(P'_m) \arrow[d, "real"]\arrow[l, "\operatorname{Ind}_{S}^{{\mathbb{K}}_{T}}{[}-2r{]}"'] \\ D^b_m(X/T) & D_m^b(X) \arrow[l, "p_!"]
\end{tikzcd}
\] Here $real$ is the realization functor and $For$ is the functor forgetting the $\mathbb{K}_T$-action.
\end{lemma}
\begin{proof} The proof of Lemma 44 in \cite{2geometric} applies verbatim to our situation. The only difference is that the second diagram in loc. cit. involves the functor $p_*$ which can be replaced by $p_!$ by Lemma \ref{*!}.
\end{proof}
\begin{proof}[Proof of Proposition \ref{sec:comp-with-soerg}] By Lemma \ref{sec:comp-with-soerg-3} applied to $X = G/U$ and adjoint $T$-action, $p^{\dagger}p_!\T_{w_0}\tw{2\dim T}$ is the image under the realization functor of the complex $\T_{w_0}\otimes_S \mathbb{K}_T$, where $\mathbb{K}_T$ is considered as a complex of free $\mathbb{K}_T^0 = S$-modules, and the action of $S$ on $\T_{w_0}$ is the antidiagonal action of the monodromy. By construction, we have $\Lambda(B_{w_0})\simeq \T_{w_0}$ and the complex $\T_{w_0}\otimes_S \mathbb{K}_T$ is the image of $K_{\mathbb{S}}$ under $h\operatorname{Ho}(\Lambda)$.
\end{proof} Let $\mathbb{K} = p^{\dagger}p_{!}\T_{w_0}\tw{2\dim T}$

The complex $\mathbb{K}$ satisfies the following properties.
\begin{lemma}
  \label{sec:comp-with-soerg-1} For any $w \in W$, we have $\DD_w\star\mathbb{K} \simeq \mathbb{K}^w(l(w)), \NN_w\star\mathbb{K} \simeq \mathbb{K}^w(-l(w)), \D'\mathbb{K} \simeq \mathbb{K}[-\dim T](-2\dim T)$. Here $\mathbb{K}^w$ stands for the complex $K$ with left monodromy action twisted by $w$.
\end{lemma}
\begin{proof} This follows from the fact that
  \[\DD_w\star\T_{w_0}\simeq\T_{w_0}\star\DD_w\simeq \T_{w_0}(l(w)),\]
  \[\NN_w\star\T_{w_0}\simeq\T_{w_0}\star\NN_w\simeq \T_{w_0}(-l(w))\] (see e.g. Proposition 7.10 of \cite{br}), Lemma \ref{sec:pro-unip-tilt-1} and an isomorphism $(K^{\bullet})^{\vee}\simeq (K^{\bullet})\tw{2\dim T}$, where $(K^{\bullet})^{\vee}$ is the Koszul complex of the diagonal bimodule over the polynomial ring $R$, see Section~\ref{sec:hochsch-cohom}.
\end{proof} Combining Proposition \ref{sec:comp-with-soerg} with Proposition \ref{sec:pro-unip-tilt} we get the following monodromic model for Hoschschild cohomology of Soergel bimodules.
Recall the notation 
\[
  \mathfrak{E}_k = \mathcal{H}^{-k}_{w_0}(\mathbb{K}).
\]
\begin{theorem}
  \label{sec:comp-with-soerg-2}There is an isomorphism of functors on $\SBim_R(W)$
  \[ \operatorname{HH}^k(M)(-l(w_0)) \simeq \underline{\Hom(\mathfrak{E}_k, \Lambda(M))^{f,\Frinv}}.
  \] More generally, for any complex $\F \in \Ho(\SBim_R(W))$ we have
  \[ \operatorname{HH}^k(\F)(-l(w_0)) \simeq \underline{\Hom(\mathfrak{E}_k, h\Ho(\Lambda)(\F))^{f,\Frinv}}.
  \]
  
\end{theorem}
In view of Proposition \ref{sec:comp-with-soerg-4} and Corollary \ref{sec:comp-with-soerg-5}, this theorem may be stated entirely in the homotopy category of Soergel bimodules. Recall the notation
\[
  \mathbb{E}_k = \mathcal{H}^{-k}_{\mathbb{S},w_0}(K_{\mathbb{S}}).
\]
\begin{corollary}
  \label{sec:comp-with-soerg-7}
  For any complex $\F \in \Ho(\SBim_R(W))$ we have the following isomorphism in $\Ho(R-\operatorname{mod})$
\[
  \operatorname{HH}^k(\F)(-l(w_0)) \simeq \Hom(\mathbb{E}_k, \F).
  \]
  In particular, for a braid $\beta$, taking $\F$ to be a Rouquier complex $F_\beta$, we get the following description of Khovanov-Rozansky homology:
  \[
\operatorname{HHH}^k({\beta})(-l(w_0)) \simeq \Hom(\mathbb{E}_k, F_\beta).
  \]
\end{corollary}
\subsection{Whittaker category.}\label{sec:whittaker-category-1} Fix a maximal unipotent subgroup $U^- \subset G$, opposite to $U$. We have an isomorphism
\[ U^{-}/[U^{-},U^{-}] \simeq \prod_{s \in S} \mathbb{G}_a,
\] where the product is of the negative root subgroups of $G$. Fix a non-trivial character $\psi:\Fq \to \Qlbar^{\times}$ and consider the corresponding Artin-Schreier local system $\operatorname{AS}_{\psi}$ on $\mathbb{G}_a$. Let $\xi$ be the composition
  \begin{equation}
  \xi: U^{-} \to U^{-}/[U^{-},U^{-}] \to \prod_{s \in S} \mathbb{G}_a \xrightarrow{+} \mathbb{G}_a,\label{eq:7}
\end{equation}
 where $+$ is the addition map. We consider the monodromic Iwahori-Whittaker category $\M_{\psi} \subset D_m^b(G/U)$ and its completion $\cM_{\psi}$ with respect to the right $T$-action on $G/U$, which, by Lemma 4.4.5 of \cite{by}, comes with the adjoint pair of functors
\[ \Av_{U!}:\cM_{\psi} \to \cM, \Av_{\psi}: \cM \to \cM_{\psi}.
\] We recall the definitions. Consider the action morphisms
\[ a_U: U \times G/U \to G/U, a_{U^-}: U^{-}\times G/U \to G/U.
\] The functor $\Av_{\psi}$ (before the completion) is defined as the composition
\[ \M \xrightarrow{\operatorname{For}} D_m^b(G/U) \xrightarrow{\operatorname{av}_{\psi}} \M_{\psi},
\] where For is the functor forgetting the left $U$-equivariant structure and
\[ \operatorname{av}_{\psi}\F = a_{U^-!}(\xi^{*}\operatorname{AS}_\psi\tw{l(w_0)} \boxtimes \F).
\] The functor $\Av_{U}$ (before the completion) is defined as
\[ \operatorname{Av}_{U}\F = a_{U!}(\underline{\Qlbar}\tw{l(w_0)} \boxtimes \F).
\]

We have the following
\begin{proposition}[\cite{by}, Corollary 5.2.4]
  \label{sec:whittaker-category}
There is an isomorphism  \[\T_{w_0} \simeq \Av_{U}\Av_{\psi}\dd.\]
\end{proposition}

\section{Character sheaves and the Harish-Chandra functor}
\label{sec:char-sheav-harish} If $X$ is a stack with an action of the algebraic group $H$, we write $\Av_H$ for the functor $\pi_!$, where $\pi: X \to X/H$ is the canonical projection. If $H\subset H'$ are two algebraic groups with $H'$ acting on $X$, we write $\Av_H^{H'}$ for the functor $\pi'_!$, where $\pi':X/H \to X/H'$ is the canonical projection. We sometimes omit $H$ from the notation, if it is clear that the sheaf to which $\Av_H^{H'}$ is applied lies in the $H$-equivariant category. Write $\operatorname{For} = \pi^*[\dim H]$. We will also sometimes omit $\operatorname{For}$ to unburden the notation. We use the notation $\Av^{\Ad}_H$ to emphasize that the action taken is adjoint, where it is relevant.

\subsection{Harish-Chandra functor.}
\label{sec:harish-chandra-funct-1}
\label{sec:harish-chandra-funct} For a subgroup $H \subset G$, let $G/_{\Ad}H$ stand for the quotient stack of $G$ by the adjoint action of $H$. We have a pair of adjoint functors $\hc_!:D_m^b(G/_{\Ad}G) \rightleftarrows D_m^b(\mathcal{Y}):\chi$ defined below.

The Harish-Chandra functor $\hc_!$ is defined as the composition
\[ D_m^b(G/_{\Ad}G) \xrightarrow{\operatorname{For}} D_m^b(G/_{\Ad}B) \xrightarrow{\Av_U} D_m^b(\mathcal{Y}),
 \] and $\chi$ is defined as the composition
\[ D_m^b(\mathcal{Y}) \xrightarrow{\operatorname{For}} D_m^b(G/_{\Ad}B) \xrightarrow{\Av^{\Ad}_G} D_m^b(G/_{\Ad}G).
 \]

For any algebraic group $G$, $D^b_m(G)$ is equipped with the $!$-convolution operation, defined as
\begin{equation}
\label{eq:2} \F \star \mathcal{G} = m_!(\F \boxtimes \mathcal{G}).
\end{equation} Define the category $D\mathcal{CS}$ -- the derived category of unipotent character sheaves -- as a full subcategory with objects $\F \in D^b_m(G/_{\Ad}G)$ such that $p^{\dagger}\hc_!(\F) \in \M$. $D\mathcal{CS}$ is a monoidal subcategory of $D^b_m(G/_{\Ad}G)$ and $p^{\dagger}\hc_!$ is a monoidal functor $D^b_m(G/_{\Ad}G) \to \M.$
\begin{bremark}
  The fact that this definition, restricted to the category of semi-simple equivariant perverse sheaves on $G$, coincides with the original definition given by Lusztig in \cite{lcs1} is the result of \cite{mirkovic1988characteristic}. 
\end{bremark}
Let $\mathcal{CS}$ stand for the full subcategory of perverse sheaves in $D\mathcal{CS}$.
\subsection{Central structure of the Harish-Chandra functor}
\label{sec:centr-struct-harish-1}
The Harish-Chandra functor is equipped with the canonical central structure. We recall what this means. Let $F: \mathcal{C} \to \mathcal{D}$ be a functor to a monoidal category $\mathcal{D}$. $F$ is said to be equipped with central structure if there is a functor $F_Z: \mathcal{C} \to Z(\mathcal{D})$ such that $F = \phi \circ F_Z$, where $Z(\mathcal{D})$ stands for the Drinfeld center of $\mathcal{D}$, and $\phi:Z(\mathcal{D})\to\mathcal{D}$ is the canonical forgetful functor. In particular, for objects $A \in \mathcal{C}, X \in \mathcal{D},$ there is an isomorphism $\beta_{A, X} : F(A) \star X \to X \star F(A) $ natural in $A, X$. Here $-\star-$ denotes the product in the monoidal category $\mathcal{D}$.

We have the following well-known
\begin{lemma}[cf. \cite{boyarchenko2014character}, Definition A.4.3]
\label{sec:centr-struct-harish} The forgetful functor $F:D^b_m(G/_{\Ad}G) \to D^b_m(G)$ carries a natural central structure with respect to the !-convolution monoidal structure on $D^b_m(G)$.
\end{lemma}
\begin{proof} We recall the construction of the commutativity constraint $\beta$, the verification of the required properties is straightforward. Let $a': G \times G \to G, a'(g,h) = g^{-1}hg$ be the (inverse) adjoint action map and let $p: G\times G \to G$ be the projection to the second factor. Define the map $\alpha:G\times G \to G\times G$ by $\alpha(g,h) = (g,a'(g,h))$. For any $\mathcal{A} \in D^b_m(G/_{\Ad}G),$ there is a canonical isomorphism $a'^*F(\mathcal{A}) \to \pi^*F(\mathcal{A})$ which gives a canonical isomorphism of complexes on $G\times G$ for an arbitrary $X \in D^b_m(G)$,
  \[ \alpha_!(X \boxtimes F(\mathcal{A})) \to X \boxtimes F(\mathcal{A}).
  \] Recall that $m: G\times G \to G$ stands for the multiplication map. Now, using the fact that $m \circ \alpha(g,h) = hg$ and applying $m_!$ to both sides of the above isomorphism, we get an isomorphism
  \[ F(\mathcal{A})\star X \simeq m_!\alpha_!(X \boxtimes F(\mathcal{A})) \to m_!\left(X \boxtimes F(\mathcal{A})\right) = X \star F(\mathcal{A}).
  \]
\end{proof} Since $U$ is isomorphic to an affine space, the category $D^b_m(U\backslash G/U)$ may be considered as a full triangulated subcategory of $D^b_m(G)$, consisting of the objects of the form $\cc_{U}\star\F\star\cc_{U}, \F \in D^b_m(G)$. We get the following
\begin{corollary}[cf. \cite{bfo}, 3.3]\label{sec:centr-struct-harish-2} The functor $p^{\dagger}\hc_!$ has a central structure.
\end{corollary}

\subsection{Pro-unit in character sheaves.} The Harish-Chandra functor constructed in the previous subsection is monoidal with respect to $!$-convolution operations \eqref{eq:2} and \eqref{eq:3} on $D^b_m(G/_{\Ad G}G)$ and $\cM$. It turns out that $\dd$ is in the image of $\operatorname{pro}D_m^b(G/_{\Ad}G)$ under the functor $p^{\dagger}\hc_!$. This is implicit in \cite{bfo} in characteristic zero setting, since the unit object of the monoidal category is naturally central, and is proved in the $\ell$-adic setting in \cite{chen}.

Let $P$ be a parabolic subgroup of $G$ with the Levi quotient $P$. There are parabolic restriction and induction functors
\[\operatorname{Res}_P^G:D^b_m(G/_{\Ad}G) \to D^b_m(L/_{\Ad}L), \operatorname{Ind}_P^G:D^b_m(L/_{\Ad}L) \to D^b_m(G/_{\Ad}G),
  \] There is also a non-equivariant version of these functors, \[\operatorname{Res}_P^G:D^b_m(G) \to D^b_m(L), \operatorname{Ind}_P^G:D^b_m(L) \to D^b_m(G),\] which we denote in the same way.

We refer the reader to \cite{lcs1}, \cite{bezrukavnikov2018parabolic}, \cite{ginzburg1993induction} and references therein for their definition and further properties.

We record some of the properties we will need:
\begin{proposition}[\cite{ginzburg1993induction}, Theorem 4.2]
  \label{sec:pro-unit-character-1} Functors $\operatorname{Res}_P^G,\operatorname{Ind}_P^G$ are $t$-exact with respect to the perverse $t$-structure. The functor $\operatorname{Ind}_P^G$ is Verdier self-dual.
\end{proposition}
\begin{proof} This is Theorem 5.4 of \cite{bezrukavnikov2018parabolic}, which generalizes the corresponding result proved for character sheaves in \cite{lcs1}, \cite{ginzburg1993induction}.
\end{proof}

Recall that according to \cite{chen_gl} Proposition 3.2, if $\F$ is a $W$-equivariant perverse sheaf on $T$, the sheaf $\ind_B^G\F$ is equipped with a $W$-action. It it will be convenient for us to use an action that differs from the action defined in loc. cit. by the sign representation of $W$. For such $\F$, let $\Phi_\F$ stand for the sheaf of $W$-invariants $(\ind_B^G\F)^W$.
 \begin{theorem}[\cite{chen}]
  \label{sec:pro-unit-character} Let $\F$ be a $W$-equivariant perverse local system on $T$. Assume that it is unipotent and corresponds to the $\Qlbar[W]\ltimes S$-module $S \otimes_{S^W} \F'$, for some $S^W$-module $\F'$. Then $\operatorname{Ind}_B^G\F$ has a natural $W$-action and $\Phi_{\F} = (\operatorname{Ind}_B^G\F)^W$ satisfies $p^{\dagger}\hc_!(\Phi_{\F}) \simeq \F$. Moreover, $\F$ is central in $\cM$.
 \end{theorem}
 \begin{proof} This follows from Theorem 7.1 together with Proposition 3.2 of \cite{chen}. Centrality property of $\F$ follows from Corollary \ref{sec:centr-struct-harish-2}.
 \end{proof}
 \begin{corollary}
   \label{centralunit} There is a family of sheaves $\mathcal{E}_n \in D_m^b(G/_{\Ad}G)$ such that \[p^{\dagger}\hc_!(\text{``$\invlim{n}$''}\mathcal{E}_n) \simeq \dd.\] Moreover, sheaves $p^{\dagger}\hc_!(\mathcal{E}_n)$ are central in the monoidal category $\M$ (and so in $\cM$).
 \end{corollary}
 \begin{proof} In \cite{chen}, Section 4.5, the construction of the sheaf $\mathcal{E}_1$ is given. It is such that, if we denote $\phi_1 = \hc_!(\mathcal{E}_1)$, $\varepsilon_1 = p^{\dagger}\phi_1$ corresponds to the $W$-coinvariants module $S/\mathfrak{m}_WS = S \otimes_{S^W} \Qlbar$ over $S$, where $\mathfrak{m}_W \subset S^W$ stands for an ideal generated by homogeneous $W$-invariant polynomials of positive degree. $\mathcal{E}_1$ is recovered as $\operatorname{Ind}_B^G(\phi_1)^W$. Sheaves $\mathcal{E}_n$ are similarly defined as $\mathcal{E}_n = \operatorname{Ind}_B^G(\phi_n)^W$, where $\varepsilon_n=p^{\dagger}\phi_n$ is the $W$-equivariant perverse local system on $T/_{\Ad}T$ corresponding to the module $S/(\mathfrak{m}_WS)^n = S \otimes_{S^W} S^W/(\mathfrak{m}_W)^n$.
 \end{proof}
 \begin{bremark}
 Note that $\mathcal{E}_1$ is, by construction, a Goresky-MacPherson extension of a certain (shifted) local system of rank $|W|$ on the set $G^{rss}$ of regular semisimple elements in $G$. It is not, however, isomorphic to the Grothendieck-Springer sheaf. For example, the sheaf $\mathcal{E}_1$ is not pure.   
 \end{bremark}
 Write formally in $\operatorname{pro}D_m^b(G/_{\Ad}G),$
\[ \dd_G = \prolim\mathcal{E}_n.
\]
\subsection{Fourier-Deligne transform and Springer action.}
\label{sec:four-deligne-transf-4} Let $E$ be an affine space of dimension $r$ defined over $\Fq$. Let $E^{\vee}$ be the dual affine space. Let $\mu: E \times E^{\vee} \to \mathbb{G}_a$ be the natural pairing, and let $\pi:E\times E^{\vee} \to E, \pi^{\vee}:E\times E^{\vee} \to E^{\vee}$ be the projections. Recall the Fourier-Deligne transform functors \[\FT_{,!}(\F), \FT_{,*}(\F): D^b_m(E) \to D^b_m(E^{\vee}),\]
\[ \FT_{,!}(\F) = \pi^{\vee}_!(\mu^*\operatorname{AS}_{\psi}\otimes\pi^*\mathcal{F})\tw{r},
\]
\[ \FT_{,*}(\F) = \pi^{\vee}_*(\mu^*\operatorname{AS}_{\psi}\otimes\pi^*\mathcal{F})\tw{r}.
 \] We write $\FT := \FT_{,!}$.

If we consider $E$ and $E^{\vee}$ as algebraic groups with respect to addition, the categories $D^b_m(E), D^b_m(E)$ become equipped with the convolution operation, as in the Section \ref{sec:harish-chandra-funct}.

For any $\xi \in E^{\vee}(\Fq)$ defining a map $\xi: E \to \mathbb{G}_a$, write $\mathcal{L}_\xi = \xi^*\operatorname{AS}_{\psi}$. We record the basic properties of the Fourier-Deligne transform.
\begin{lemma}[\cite{katz1985transformation}, \cite{brylinski1986transformations}]
 \label{sec:four-deligne-transf}
  \begin{enumerate}[label=(\arabic*)]
 \item \label{item:4} The natural transformation given by forgetting the support  $\FT_{,!} \to \FT_{,*}$ is an isomorphism.
  \item \label{item:5} $\FT$ is an equivalence of triangulated categories.
  \item \label{item:6} $\FT$ sends perverse sheaves to perverse sheaves, and sends pure complexes of weight $w$ to pure complexes of weight $w$.
  \item
    \label{item:9} For any $\mathcal{A},\mathcal{B} \in D^b_m(E)$, we have
    \[ \FT(\mathcal{A}\otimes_{\Qlbar}\mathcal{B}) \simeq \FT(\mathcal{A})\star\FT(\mathcal{B}).
    \]
  \item
\label{item:10} $\FT(\mathcal{L}_\xi\tw{r}) \simeq \delta_\xi$, where $\delta_\xi$ stands for the rank one punctual sheaf at $\xi$.
  \end{enumerate}
\end{lemma}
\begin{proof}  \ref{item:4} is \cite{katz1985transformation} Theorem 2.4.1, \ref{item:5} follows directly from \cite{brylinski1986transformations} Proposition 9.3. \ref{item:6} is \cite{katz1985transformation} Corollary 2.1.5 and Theorem 2.2.1. \ref{item:9} is \cite{brylinski1986transformations} Corollary 6.3. \ref{item:10} is a straightforward computation.
\end{proof}

Let $\g$ be the Lie algebra of $G$, $\mathfrak{b} \subset \g$ the Lie algebra of $B$. Fix a $G$-invariant non-singular bilinear form $\langle , \rangle$ on $\g$, thus identifying the affine space underlying $\g$ with its dual.

Let $\g^{rss}$ denote the open subset of regular semisimple elements of $\g$.

Let $\tilde{\g} = G \times^B \mathfrak{b}$ be the Grothendieck-Springer resolution, with the standard projection $p: \tilde{\g} \to \g$, and denote by $\tilde{\g}^{rss}$ the preimage of $\g^{rss}$ under $p$. It is well-known that $\g^{rss}$ is a Galois cover of $\g^{rss}$ with the Galois group $W$. Moreover, $p$ is a small map, so that the Grothendieck-Springer sheaf $\mathfrak{gspr}: = p_*\cc_{\g}\tw{\dim \g}$ is identified with the Goresky-MacPherson extension of its restriction to $\g^{rss}$. We have an identification \( \Qlbar [W] \simeq \operatorname{End}(\mathfrak{gspr}), \) which, in turn, defines a functor $S_{\g}:\Rep_{\Qlbar}(W) \to D^b_m(\g)$ from the category of finite dimensional representations of $W$ over $\Qlbar$.

Let $\mathfrak{n} \subset \mathfrak{b}$ be the Lie algebra of $U$, and let $\mathcal{N}_{\g}$ be the nilpotent variety of $\g$. Let $\tilde{\mathcal{N}_{\g}} = G \times^B \mathfrak{n}$ be the Springer resolution, and $q: \tilde{\mathcal{N}_{\g}} \to \mathcal{N}_{\g}$ be the standard projection. Denote by $\mathfrak{spr}:= q_*\cc_{\tilde{\mathcal{N}_{\g}}}\tw{\dim \mathcal{N}_{\g}}$ the Springer sheaf. We have $\FT(\mathfrak{gspr}) \simeq \mathfrak{spr}$, which defines $W$-action on $\mathfrak{spr}$ and the functor $S_{\mathcal{N}_{\g}}: \Rep_{\Qlbar}(W) \to D^b_m(\mathcal{N}_{\g})$. This action of $W$ on the Springer sheaf was constructed in \cite{brylinski1986transformations}. It is easy to see that we have $S_{\mathcal{N}_{\g}}(\operatorname{triv}_W) = \iota_{0*}\cc_{\Spec(\Fq)}$, where $\operatorname{triv}_W$ stands for the trivial representation of $W$, and $\iota_{0}:\Spec(\Fq)\to \g$ is the inclusion of $\{0\}$ to $\g$.

There is another action of $W$ on $\mathfrak{spr}$ obtained from the base change isomorphism $\iota_{\mathcal{N}_{\g}}^*\mathfrak{gspr}[-\dim T]\simeq\mathfrak{spr}$, considered in \cite{borho1981representations}. Let $S'_{\mathcal{N}_{\g}}: \Rep_{\Qlbar}(W) \to D^b_m(\mathcal{N}_{\g})$ be the corresponding functor. Let $\operatorname{sgn}_W$ stand for the sign representation of $W$. We have the following
\begin{theorem}[Theorem 1.1 of \cite{achar2014weyl}]
  \label{sec:four-deligne-transf-1} Two actions of $W$ on $\mathfrak{spr}$ differ by the sign representation of $W$: we have a natural isomorphism
 \[ S_{\mathcal{N}_{\g}}(-) = S'_{\mathcal{N}_{\g}}(-\otimes \operatorname{sgn}_W).
 \]
\end{theorem}
\begin{corollary}
 \label{sec:four-deligne-transf-2} We have $S_{\mathcal{N}_{\g}}(\operatorname{sgn}_W) \simeq \cc_{\mathcal{N}_{\g}}\langle\dim \mathcal{N}_{\g}\rangle$.
\end{corollary}

Since the kernel of the Fourier-Deligne transform on $\g$ is $G$-equivariant with respect to the diagonal adjoint action, one can similarly define an equivariant version $\FT:D^b_m(\g/_{\Ad}G) \to D^b_m(\g/_{\Ad}G)$. It is easy to see that it commutes with the !-$\Ad_G$ averaging.

Most of the constructions described above have a group version.

Let $\tilde{G} = G \times^B B$ be the group version of the Grothendieck-Springer resolution, with the standard projection $p': \tilde{G} \to G$, and denote by $\tilde{G}^{rss}$ the preimage of $G^{rss}$ under $p'$. It is well-known that $\tilde{G}^{rss}$ is a Galois cover of $G^{rss}$ with the Galois group $W$. Moreover, $p'$ is a small map, so that the Grothendieck-Springer sheaf $\operatorname{GSpr} : = p'_*\cc_{\g}\tw{\dim \g}$ is identified with the Goresky-MacPherson extension of its restriction to $G^{rss}$. We have an identification \( \Qlbar [W] \simeq \operatorname{End}(\operatorname{GSpr}), \) which, in turn, defines a functor $S_{G}:\Rep_{\Qlbar}(W) \to D^b_m(G)$ from the category of finite dimensional representations of $W$ over $\Qlbar$.

 Let $\mathcal{N}_G$ stand for the variety of unipotent elements in $G$. $\tilde{\mathcal{N}_{G}} = G \times^B U$ be the Springer resolution, and $q: \tilde{\mathcal{N}_{G}} \to \mathcal{N}_{G}$ be the standard projection. Denote by $\operatorname{Spr}:= q_*\cc_{\tilde{\mathcal{N}_{G}}}\tw{\dim \mathcal{N}_{G}}$ the Springer sheaf. Completeley analogous to the case of $\g$, we have an action of $W$ on $\operatorname{Spr}$ obtained from the base change isomorphism $\iota_{\mathcal{N}_{G}}^*\operatorname{GSpr}[-\dim T]\simeq\operatorname{Spr}$, and the corresponding functor $S'_{\mathcal{N}_{G}}:\Rep_{\Qlbar}(W) \to D^b_m(G).$

 For a good prime $p$, there are $G$-equivariant isomorphisms
\[\varphi:\mathcal{N}_\g\to\mathcal{N}_G, \tilde{\varphi}:\tilde{\mathcal{N}}_\g\to\tilde{\mathcal{N}}_G,\] making the following diagram commutative
\begin{equation}
  \label{sec:four-deligne-transf-3}
  \begin{tikzcd} \tilde{\mathcal{N}}_\g\arrow[r, "\tilde{\varphi}"] \arrow[d, "p"']& \tilde{\mathcal{N}}_G \arrow[d, "p'"]\\ {\mathcal{N}}_\g\arrow[r, "\varphi"] & {\mathcal{N}}_G \\
   \end{tikzcd}
\end{equation} See \cite{humphreys2011conjugacy} Chapter 6.20 and references therein.

Base change formula gives the following formula for the composed functor $\chi\circ\hc_!$:
\begin{lemma}[\cite{gi} Lemma 8.5.4]
  \label{sec:four-deligne-transf-5}
  We have an isomorphism of functors
  \[
    \chi\hc_!(-) \simeq \operatorname{Spr} \star (-).
  \]
\end{lemma}
This equips the functor $\chi\hc_!$ with $W$-action coming from the $W$-action on $\operatorname{Spr}$.  
We have the following compatibility result.
\begin{lemma} The $W$-action on $\mathfrak{spr}$ corresponding to the functor $S'_{\mathcal{N}_{\g}}$ coincides with the base change along the diagram \eqref{sec:four-deligne-transf-3} of the $W$-action on $\operatorname{Spr}$ corresponding to the functor $S'_{\mathcal{N}_G}$.
\end{lemma}
\begin{proof} By Theorem 4.8 (1) of \cite{achar2014weyl}, it is enough to check that the induced actions on the cohomology of the fiber of $p$ (respectively $p'$) over $0 \in \g$ (respectively $e \in G$) are the same, which is a straightforward verification.
\end{proof}

Let $S_{\mathcal{N}_G} = S_{\mathcal{N}_G}' \otimes \operatorname{sgn}_W$.

We also record the following compatibility result for future use.
\begin{proposition}
\label{sec:four-deligne-transf-6}
  Let $\F\in D^b_m(T)$ be a $W$-equivariant perverse sheaf satisfying the assumptions of Theorem \ref{sec:pro-unit-character}, so that $\hc_!(\Phi_\F) \simeq \operatorname{Res}_B^G(\Phi_\F)$. Two $W$-actions on the perverse sheaf $Ind_B^G(\F)$, one coming from the construction of \cite{chen_gl}, see the discussion before Theorem \ref{sec:pro-unit-character}, and another coming from the isomorphism
  \[ \ind_B^G\res_B^G(\Phi_\F) \simeq \chi\hc(\F) \simeq \operatorname{Spr}\star\F,
  \] and the $W$-action on $\operatorname{Spr}$ defined above, coincide.
\end{proposition}
\begin{proof} By Proposition 3.2 (1) of \cite{chen}, for any complex $A$ on $D^b_m(T)$, we have an isomorphism
  \[ \oplus_{w \in W}w^*A \simeq \operatorname{Res}_B^G\operatorname{Ind}_B^G(A).
  \] This implies that the functor $\res_B^G$ is faithful on the essential image of $\ind_B^G$ in $D^b_m(G)$, and hence its equivariant version $\res_B^G: D^b_m(T/_{\Ad}T) \to D^b_m(G/_{\Ad}G)$ is faithful on the essential image of the restriction of $\ind_B^G$ to the subcategory of perverse equivariant sheaves.

  It follows that it is enough to check the compatibility of $W$- actions on
  \[ \res_B^G\ind_B^G\res_B^G(\Phi_\F).
  \] By Proposition 3.2 (2) of \cite{chen}, for any $W$-equivariant complex $A$ on $D^b_m(T)$, we have an isomorphism
  \[ \Qlbar[W]\otimes A \simeq \operatorname{Res}_B^G\operatorname{Ind}_B^G(A),
  \] compatible with $W$-action. Applying this to $A = \F \simeq \res_B^G(\Phi_\F)$, we get that it is enough to check the compatibility of $W$-actions on $\Qlbar[W]\otimes\res_B^G(\Phi_\F)$ and $\res_B^G(\operatorname{Spr}\star\Phi_\F)$.

    To do this, we note that we have, for $A\in D^b_m(G)$ satisfying $\hc_!(A) \simeq \operatorname{Res}_B^G(A)$, canonical natural isomorphisms of functors
    \[ \res_B^G(-\star A)\simeq \res_B^G(-)\star\res_B^G(A),
    \] and
    \[ \ind_B^G(-\star\res_B^GA)\simeq \ind_B^G(-)\star A.
    \] Here the first property follows straightforwardly by base change and $\hc$ being a monoidal functor, and the second property is Proposition 3.3 in \cite{chen}.

    It follows that it is enough to check the compatibility of $W$-actions for $\F = \delta_e$ a skyscraper sheaf at the unit $e \in T$, for which it is true by definition.
\end{proof}
\subsection{Whittaker averaging.}
See Section \ref{sec:whittaker-category-1} for notations.
Further properties of the Whittaker averaging of conjugation equivariant sheaves will be described in a forthcoming paper \cite{bezrukavnikovdeshpande}.

Recall that $\operatorname{AS}_{\psi}$ stands for the Artin-Schreier local system on $\mathbb{G}_a$ 

 Let $i_{-}:U^{-} \to G$ be the embedding map.

Write $\Xi_\psi := \Av_{G}^{\Ad}i_{-*}\xi^*\operatorname{AS}_{\psi}\tw{2l(w_0)}$.

\begin{theorem}
  \label{sec:whittaker-averaging-5}
  There is an isomorphism
   \[ \hc_!(\Xi_\psi\star\dd_G) \simeq p_!\T_{w_0}[\dim T](l(w_0)).
   \]
 \end{theorem}
 \begin{proof} Recall from Corollary \ref{centralunit}, that
   \[ p^{\dagger}\hc_!(\dd_G)\simeq p^{\dagger}\hc_!(\prolim \mathcal{E}_n) = \prolim \varepsilon_n.
   \]

   By Proposition \ref{sec:whittaker-category}, we need to prove that
  \[ \hc_!(\Av_{G}^{\Ad}i_{-*}\xi^*\operatorname{AS}_{\psi}\star\mathcal{E}_n) \dot{=} \Av_T^{\Ad}\Av_U\Av_\psi\varepsilon_n.
 \] Rewrite the right-hand side as
 \[ \Av_T^{\Ad}\Av_U\Av_\psi\varepsilon_n \dot{=} \Av_T^{\Ad}\Av_U(i_{-*}\xi^*\operatorname{AS}_\psi)\star\hc_!(\mathcal{E}_n).
 \] Here we used the fact that the Whittaker averaging can be expressed as convolution with $i_{-*}\xi^*\operatorname{AS}_\psi$, see Section \ref{sec:whittaker-category-1}:
 \[ \Av_\psi(-) \simeq \operatorname{For}(i_{-*}\xi^*\operatorname{AS}_\psi \star -)
 \] and that $\varepsilon_n \dot{=} p^{\dagger}\hc_!(\mathcal{E}_n)$.

 For the left-hand side we have
 \begin{multline*} \hc_!(\Av_{G}^{\Ad}i_{-*}\xi^*\operatorname{AS}_{\psi}\star\mathcal{E}_n) \simeq \hc_!(\Av_{G}^{\Ad}i_{-*}\xi^*\operatorname{AS}_{\psi})\star\hc_!(\mathcal{E}_n) \dot{=}\\\dot{=} \Av_U(\Av_B^{\Ad}i_{-*}\xi^*\operatorname{AS}_{\psi})\star\hc_!(\mathcal{E}_n)\dot{=} \Av_T^{\Ad}\Av_U(i_{-*}\xi^*\operatorname{AS}_\psi)\star\hc_!(\mathcal{E}_n).
 \end{multline*} Here we used that $i_{-*}\xi^*\operatorname{AS}_\psi$ is $\Ad U^-$-equivariant, and that \[\Av_U\Av_B^{\Ad} \dot{=} \Av_T^{\Ad}\Av_U.\]

 Shift and twist can now be recovered by considering the stalk at the generic point.
\end{proof}
Let $\mathcal{N}_G^{reg} \subset \mathcal{N}_G^{reg}$ be the regular unipotent orbit of the adjoint $G$-action. Let $j_{reg}:\mathcal{N}_G^{reg}/_{\Ad}G \to G/_{\Ad}G$ be the corresponding embedding.

 \begin{proposition}
   \label{sec:whittaker-averaging-2} Assume that $G$ is of adjoint type. Then
   \[ \Xi_\psi \simeq j_{reg*}\underline{\Qlbar}_{\mathcal{N}_G^{reg}}\tw{2l(w_0)}.
   \]
 \end{proposition}
 \begin{proof} To unburden the notations, we will write $\cc$ instead of $\cc_X$ for a constant sheaf on the stack $X$, where it does not cause confusion. We have $\Av_{G}^{\Ad} \simeq \Av_{B^-}^{G}\Av_{B^-}^{\Ad}$. We first compute the averaging with respect to $B^-$. Note that $\xi$ is $\Ad U^-$-equivariant, so it is enough to compute $\Av_{T}^{\Ad}i_{-*}\xi^*\operatorname{AS}_{\psi}$.  Since the projection $U^- \to U^-/[U^-,U^-]$ is $\Ad T $-equivariant, it is enough to compute the averaging of $+^*\operatorname{AS}_{\psi}$, where $+$ stands for the addition map from \eqref{eq:7}. This, in turn, reduces to the one-dimensional computation, namely of the averaging of $\operatorname{AS}_{\psi}$ with respect to the scaling action of $\mathbb{G}_m$. To do this, consider the diagram
   \[
     \begin{tikzcd} \mathbb{G}_m\times\mathbb{G}_a \arrow[r, "\pi_2"] \arrow[rd, "m"] & \mathbb{G}_a \\ & \mathbb{G}_a
\end{tikzcd}
\] Here $m$ is the action (multiplication) morphism, and $\pi_2$ is the projection to the second factor. We wish to compute the sheaf $m_!\pi_2^*\operatorname{AS}_\psi$. The diagram above is isomorphic, via the map $(a,b) \mapsto (a^{-1},ab), \mathbb{G}_m\times\mathbb{G}_a \to \mathbb{G}_m\times\mathbb{G}_a,$ to the diagram
   \[
     \begin{tikzcd} \mathbb{G}_m\times\mathbb{G}_a \arrow[r, "m"] \arrow[rd, "\pi_2"] & \mathbb{G}_a \\ & \mathbb{G}_a
\end{tikzcd}
\] which fits into the following larger diagram
\[
\begin{tikzcd} \mathbb{G}_m\times\mathbb{G}_a \arrow[r, "j'"] \arrow[d, "\pi_1"] & \mathbb{G}_a\times\mathbb{G}_a \arrow[r, "m'"] \arrow[rd, "\pi_2"] \arrow[d, "\pi_1"] & \mathbb{G}_a \\ \mathbb{G}_m \arrow[r, "j"] & \mathbb{G}_a & \mathbb{G}_a
\end{tikzcd}
\] Here $j,j'$ are the canonical open embeddings, $\pi_1, \pi_2$ are projections to the first and second factors, respectively, and $m'$ is the multiplication map. It is easy to see that we have a natural isomorphism
\[ \pi_{2!}m^*\operatorname{AS}_{\psi} \simeq \pi_{2!}j'_!(j'^*m'^*\operatorname{AS}_\psi) \simeq \pi_{2!}(\pi_1^*j_!\underline{\Qlbar}\otimes m'^*\operatorname{AS}_\psi),
\] and so we need to compute the Fourier-Deligne transform of $j_!\underline{\Qlbar}$, which can be easily seen to be isomorphic to $j_*\underline{\Qlbar}(1)$. Thus, we get that
\[ \Av_{T}^{\Ad}\xi^*\operatorname{AS}_\psi\tw{\dim T} \simeq j^{-,reg}_*\underline{\Qlbar}(\dim T),
\] where $j^{-,reg}$ is the open embedding of the subset of regular elements in $U^-$. The complex $j^{-,reg}_*\underline{\Qlbar}(\dim T)$ is $\Ad B^- = \Ad U^-T$-equivariant, and, after averaging from $B^-$ to $G$ (note such an operation is given by a proper pushforward) we get the result.
\end{proof}
We have the following immediate corollary for an arbitrary connected reductive group $G$. Let $Z(G)$ be the center of $G$.  
\begin{corollary}
  \label{sec:whittaker-averaging-8}
  There is an isomorphism
\[
  \Xi_\psi \simeq \Av_{Z(G)}j_{reg*}\underline{\Qlbar}_{\mathcal{N}_G^{reg}}\tw{2l(w_0)}.
\]
\end{corollary}
\subsection{Perverse cohomology of $\Xi_{\psi}$ and Hochschild cohomology.} We now compute the perverse cohomology of the sheaf $\Xi_\psi$. Recall from Section \ref{sec:four-deligne-transf-4} that the irreducible constituents of the Springer sheaf are labeled by the irreducible representations of $W$, so that the IC-sheaf of $\mathcal{N}_G^{reg}$ corresponds to the sign representation. Let $\mathfrak{t}$ be the representation of $W$ on $H^1(T,\Qlbar)$, and let \[\mathfrak{W}_k = S_{\mathcal{N}_G}({\Lambda^k \mathfrak{t}}).\] Let $n = \dim T$.
\begin{theorem}
  \label{sec:whittaker-averaging-1}
  There is an isomorphism
\[ \mathcal{H}^{k}(\Xi_\psi) \simeq \mathfrak{W}_{n-k}(-2k), k = 0,\dots, n.
\]
\end{theorem} To prove Theorem \ref{sec:whittaker-averaging-1}, we shall express the Fourier transform of $\Xi_\psi$ as an averaging of a constant sheaf on the Kostant slice.

  We recall several facts about the Kostant slice in large characteristic, see \cite{riche2017kostant}.

  Let $e$ be a regular nilpotent element dual to the character $\psi$. Let $f, h \in \mathfrak{g}$ be such that $[h,e] = 2e, [h,f]=-2f, [e,f] = h$ (such a pair exists by the Jacobson-Morozov theorem). Let $\mathcal{S} = e + \ker \operatorname{ad}f$ be the Kostant slice to $e$ in $\g$.

  Let $\mathfrak{b} \subset \g$ be the Borel subalgebra containing $e$, and let $\mathfrak{b}^-$ be an opposite Borel subalgebra. Let $U^-$ be the unipotent radical of the Borel subgroup with Lie algebra $\mathfrak{b}^-$. Let $\mathfrak{n}^-\subset \mathfrak{b}^-$ be its nilpotent radical.

  Let $a,\pi_2: G \times \g \to \g $ be the adjoint action and projection maps, respectively, let $\delta_\g: \g \to \g \times \g$ be the diagonal embedding, and let $\iota_{\mathcal{S}}$ be a closed embedding of $\mathcal{S}$ to $\g$. Consider the following diagram, where both squares are Cartesian:

  \begin{equation}
\label{eq:1}
    \begin{tikzcd} \mathcal{I_S} \arrow[d] \arrow[r] & \mathcal{I} \arrow[d] \arrow[r] & G\times \g \arrow[d, "a \times \pi_2"] \\ \mathcal{S} \arrow[r, "\iota_{\mathcal{S}}"] & \g \arrow[r, "\delta_\g"] & \g \times \g
    \end{tikzcd}
  \end{equation}
  
  The scheme $\mathcal{I}$ is known as the universal centralizer. It is easy to see from the definition that the scheme-theoretic fiber of $\mathcal{I}$ over $x \in \g$ is the centralizer $G_x$ of $x$ in $G$.

  The following Lemma summarizes some properties of $\mathcal{S}$ we will use.
  \begin{lemma}
\label{sec:perv-cohom-xi_psi}
    \begin{enumerate}[label=(\arabic*)]
    \item
      \label{item:7} The following diagram, where the unlabeled maps are restrictions of those in \eqref{eq:1}, is Cartesian.

      \[
        \begin{tikzcd}
          \mathcal{I_S} \arrow[d] \arrow[r] & G\times \mathcal{S} \arrow[d, "a"] \\ \mathcal{S} \arrow[r, "\iota_{\mathcal{S}}"] & \g
        \end{tikzcd}
      \]
      
  \item
    \label{item:8} The adjoint action map $U^- \times \mathcal{S} \to e+\mathfrak{b}^-$ is an isomorphism.
    \end{enumerate}
  \end{lemma}
  \begin{proof}[Proof of the Lemma.] For \ref{item:7} see \cite{riche2017kostant}, proof of Lemma 3.3.5. For \ref{item:8} see \cite{gan2002quantization}, Lemma 2.1.
  \end{proof}
  Let $D_{\operatorname{Spr}}\subset D^b(\mathcal{N}_G/G)$ be the full triangulated subcategory category generated by the summands of $\operatorname{Spr}$.
\begin{proof}[Proof of the Theorem \ref{sec:whittaker-averaging-1}] Under our assumption on $p$, we can choose a $G$-equivariant identification $\mathcal{N}_\g \to \mathcal{N}_G$ and the compatible identification of $\mathfrak{n}^-$ with $U^-$. For the rest of the proof, we will be working with $\mathcal{N}_\g$, denoting in the same way objects over $\mathcal{N}_G$ pulled back via this identification. 

  By the result of \cite{rider2016perverse} Theorem 3.5 (see also references therein), all perverse cohomology of $\omega j_{reg*}\underline{\Qlbar}_{\mathcal{N}_G^{reg}}$ are in the category $D_{\operatorname{Spr}}$ generated by the summands of the Springer sheaf, since $\omega j_{reg*}\underline{\Qlbar}_{\mathcal{N}_G^{reg}}$ is indecomposable. By Corollary \ref{sec:whittaker-averaging-8} the same is true for $\Xi_\psi$. Hence, to identify $\mathcal{H}^{k}(\Xi_\psi)$, it is enough to identify the restriction of its Fourier transform to $\g^{rss}$.

  Note that, by Lemma \ref{sec:four-deligne-transf}, \ref{item:9} and \ref{item:10}, we have
  \[ \FT(\mathcal{L}_\psi^{\mathfrak{n}^-}\tw{\dim \mathfrak{n}^-})\simeq \cc_{e+\mathfrak{b}^-}\tw{\dim\mathfrak{b}^-}.
  \]
  We have
\[ \FT(\Xi_\psi) \dot{=} \FT(\Av_G^{\Ad}\mathcal{L}_\psi^{\mathfrak{n}^-}) \dot{=} \Av_G^{\Ad}\cc_{e+\mathfrak{b}^-}.\ 
\] By Lemma \ref{sec:perv-cohom-xi_psi},
\[ \iota_{\mathcal{S}}^*\Av_G^{\Ad}\cc_{e+\mathfrak{b}^-} \dot{=} \pi_{\mathcal{S}_!}\cc_{\mathcal{I_S}},
\] where $\pi_{\mathcal{S}}:\mathcal{I_S} \to \mathcal{S}$ is the canonical projection. It is easy to see that the stalk of $\pi_{\mathcal{S}_!}\cc_{\mathcal{I_S}}$ over $x \in \g^{rss}$ is $W$-equivariantly identified with $H_c^\bullet(T) \simeq \wedge^{2n-\bullet}\mathfrak{t}$, hence the result.
\end{proof} 

The above theorem allows us to compute the perverse cohomology of $\mathbb{K}.$
\begin{lemma}
 \label{sec:whittaker-averaging} Let $\F \in \mathcal{CS}$ be such that $\hc_!(\F)$ is supported on $T$. Then the functor of convolution with $\F$ is a t-exact functor $D_{\operatorname{Spr}} \to \mathcal{CS}$.
\end{lemma}
\begin{proof}  It is enough to show that $\F\star\operatorname{Spr}$ is perverse. We have, by Lemma \ref{sec:four-deligne-transf-5},
  \[ \F \star\operatorname{Spr} \simeq \chi\hc_!\F = \chi\operatorname{Res}^G_B\F,
  \] since $\operatorname{Res}^G_B\F = \hc_!\F$ for sheaves $\F$ such that $\hc_!\F$ is supported on $T$. On the other hand, $\chi$ coincides with the t-exact parabolic induction functor when applied to $\Ad T$-equivariant sheaves supported on $T$, so $\F \star\operatorname{Spr}$ is perverse.
\end{proof}

We will use the following result, proved in characteristic 0 setting in \cite{bfo} and in $\ell$-adic setting in \cite{chenyd}.

Recall the notations of Definition \ref{sec:comp-with-soerg-6}. Note that for any $w \in W$, objects $\DD_w, \NN_w$ descend to the stack $\mathcal{Y}_w$. It follows that, if $\F\in\cM$ descends to an object on the stack $\mathcal{Y}_v, v \in W$, the objects $\DD_{w}\star \F, \NN_{w}\star \F$ descend to the objects on the stack $\mathcal{Y}_{wv}$. It follows that convolution with $\DD_{w},\NN_{w}$ define functors $D^b_m(\mathcal{Y}_v) \to D^b_m(\mathcal{Y}_{wv})$, which we denote by $\DD_w\star -, \NN_w\star -$, abusing notation.
\begin{proposition}[\cite{chenyd}]
\label{sec:whittaker-averaging-3} The functor \[\F \mapsto \NN_{w_0}\star\hc_!(\F), D\mathcal{CS} \to D^b_m(\yy_{w_0})\] is t-exact and commutes with Verdier duality.
\end{proposition}
\begin{corollary}
\label{sec:perv-cohom-xi_psi-2}
  We have \[\mathcal{H}^{k}(\NN_{w_0}\star p^{\dagger}p_!\T_{w_0}[n](l(w_0)) \simeq \NN_{w_0}\star p^{\dagger}\hc_!(\mathfrak{W}_{n-k}(-2k)\star\dd_G).\]
\end{corollary}
\begin{proof} This is a direct corollary of Theorem \ref{sec:whittaker-averaging-5}, Theorem \ref{sec:whittaker-averaging-1}, Proposition \ref{sec:whittaker-averaging-3} and Lemma \ref{sec:whittaker-averaging}.
\end{proof} Finally, we compute the objects representing $k$-th Hochschild cohomology in the monodromic category.

Recall the notation 
\[
  \mathfrak{E}_k = \mathcal{H}^{-k}_{w_0}(\mathbb{K}).
\]
\begin{theorem}
  \label{sec:whittaker-averaging-4}
  There is an isomorphism
  \[ \mathfrak{E}_k \simeq \DD_{w_0}\star\D'(\NN_{w_0}\star p^{\dagger}\hc_!(\mathfrak{W}_{n-k}\star\dd_G)(-2k+l(w_0))).
  \]
\end{theorem}
\begin{proof} By Definition \ref{sec:pro-unip-tilt-2} of the shifted t-structure, we have
  \[ \mathcal{H}^{-k}_{w_0}(\mathbb{K}) = \DD_{w_0}\star\mathcal{H}'^k(\NN_{w_0}\star\mathbb{K}) \simeq \DD_{w_0}\star\D'\mathcal{H}^{k}(\D'(\NN_{w_0}\star\mathbb{K})).
  \]
  Applying Lemma \ref{sec:comp-with-soerg-1} three times, we get that the latter is isomorphic to
\begin{multline*} \DD_{w_0}\star\D'\mathcal{H}^{k}(\D'(\mathbb{K}^{w_0}(-l(w_0))) \simeq \\ \simeq \DD_{w_0}\star\D'\mathcal{H}^{k}(\mathbb{K}^{w_0}[-n](-2n + l(w_0)))\simeq\\ \simeq \DD_{w_0}\star\D'\mathcal{H}^{k}(\NN_{w_0}\star \mathbb{K}[-n](-2n + 2l(w_0))).
\end{multline*} By Corollary \ref{sec:perv-cohom-xi_psi-2} , the last expression is given by
\[ \DD_{w_0}\star\D'(\NN_{w_0}\star p^{\dagger}\hc_!(\mathfrak{W}_{n-k}\star\dd_G)(-2k+l(w_0))).
\]
     \end{proof}
\subsection{Character sheaves and the full twist.} 
In this section we simplify the expression for $\mathfrak{E}_k$ given in Theorem \ref{sec:whittaker-averaging-4} and also prove several results about the pro-object in $D\mathcal{CS}$ sent to the full twist object $\DD_{w_0}^2 = \DD_{w_0}\star \DD_{w_0}$ by the functor $\hc_!$. Most of the section is occupied with a proof of the following result.
     \begin{proposition}
   \label{sec:whittaker-averaging-9} There is an isomorphism
       \[ p^{\dagger}\hc_!(\mathfrak{W}_k\star\dd_G) \simeq \DD_{w_0} \star \D'(\NN_{w_0}\star p^{\dagger}\hc_!(\mathfrak{W}_{n-k}\star\dd_G)).
       \]
     \end{proposition} To prove Proposition \ref{sec:whittaker-averaging-9} we will use the following algebraic statement. Recall some notations from the proof of Corollary \ref{centralunit}. Let $\mathfrak{m}, \mathfrak{m}_W$ be the maximal graded ideals of $S$ and $S^W$, respectively, and let $I$ be the ideal $\mathfrak{m}_WS$ of $S$. For a finite-dimensional $\Qlbar[W]\ltimes S$-module (respectively, $S^W$-module) $M$ let $M^{\vee}$ be the dual $\Qlbar[W]\ltimes S$-module (respectively, $S^W$-module). 
     \begin{lemma}\label{sec:char-sheav-full} Let $C_n = S/I^n \simeq S \otimes_{S^W} S^W/\mathfrak{m}_W^n$. We have an isomorphism of graded $\Qlbar[W]\ltimes S$-modules
       \[ C_n^{\vee} \simeq \operatorname{sign}_W \otimes_{\Qlbar} S \otimes_{S^W}(S^W/\mathfrak{m}_W^n)^{\vee}(2l(w_0)).
       \]
     \end{lemma}
     \begin{proof}
       Let $X = \Spec(S), Y = \Spec(S^W), p:X \to Y$ be the map corresponding to the embedding $S^W \to S$ and let $M = S^W/\mathfrak{m}_W^n$ considered as an $S_W$-module. Let $D$ stand for the Grothendieck-Serre duality functor. We have
       \[
         DC_n = Dp^*M \simeq p^!DM,
       \]
       with $p^!DM \simeq \Hom_{S^W}(S, DM)$. Note that $S$ is a free $S^W$-module equipped with a $W$-invariant perfect paring over $S^W$ with values in \[\operatorname{sgn}_W\otimes_{\Qlbar} S^W(-2l(w_0)),\] and so
       \[
         \Hom_{S^W}(S, DM) \simeq \operatorname{sgn}_W\otimes_{\Qlbar} S \otimes_{S^W}DM.
       \]
     \end{proof}
Let $\varepsilon'_n$ be the $W$-equivariant perverse local system corresponding to the module $\sgn_W\otimes C_n^{\vee}$ from Lemma~\ref{sec:char-sheav-full}, cf. also Corollary \ref{centralunit}. By Theorem \ref{sec:pro-unit-character}, there is an $\Ad G$-equivariant perverse sheaf $\mathcal{E}_n' = (\operatorname{Ind}_B^G\varepsilon'_n)^W$ on $G$, such that $p^{\dagger}\hc_!\mathcal{E}_n' = \varepsilon_n'$. For an arbitrary representation $V$ of $W$ write $\mathcal{E}_n^V, \mathcal{E}'^V_n$ for the isotypic component $\left(\operatorname{Ind}_B^G\varepsilon_n\right)^V, \left(\operatorname{Ind}_B^G\varepsilon'_n\right)^V$ of $V$ in $\operatorname{Ind}_B^G\varepsilon_n, \operatorname{Ind}_B^G\varepsilon'_n$, respectively.
     \begin{proposition}
       \label{sec:char-sheav-full-1}
There are isomorphisms       
\[
  \D\mathcal{E}_n'^V \simeq \mathcal{E}_n^{\sgn_W\otimes V}\simeq S_{\mathcal{N}_G}(\sgn_W\otimes V)\star\mathcal{E}_n.
\]
     \end{proposition}
     \begin{proof}
        By Lemma~\ref{sec:char-sheav-full} we have
       \begin{equation}
         \label{eq:5}
 \operatorname{Ind}_B^G\D\varepsilon_n' \simeq \operatorname{Ind}_B^G(\varepsilon_n\otimes\operatorname{sign}_W(-2l(w_0)-2\dim T)).
\end{equation}
       We also have
       \begin{equation}
         \label{eq:6}
 \operatorname{Ind}_B^G\varepsilon_n \simeq \operatorname{Ind}_B^G\operatorname{Res}_B^G\mathcal{E}_n \simeq \chi\hc_!\mathcal{E}_n \simeq \operatorname{Spr}\star\mathcal{E}_n,
\end{equation}

with $W$-action on $\operatorname{Ind}_B^G\varepsilon_n$ compatible with the $W$-action on $\operatorname{Spr}$ by Proposition~\ref{sec:four-deligne-transf-6}. Since the Verdier duality $\D$ commutes with the parabolic induction, we have 
\[\D\mathcal{E}_n'^V \simeq \D(\operatorname{Ind}_B^G\varepsilon'_n)^V \simeq (\operatorname{Ind}_B^G\D\varepsilon'_n)^V.\]
By \eqref{eq:5}, up to a Tate twist, we have
\[
(\ind_B^G\D\varepsilon'_n)^V \dot{=} \left(\operatorname{Ind}_B^G(\varepsilon_n\otimes\operatorname{sign}_W)\right)^V \simeq\left( \operatorname{Ind}_B^G\varepsilon_n\right)^{\operatorname{sgn}_W\otimes V}.
\]
By \eqref{eq:6} and Proposition \ref{sec:four-deligne-transf-6},
\[\left(\operatorname{Ind}_B^G\varepsilon_n\right)^{\operatorname{sgn}_W\otimes V}\simeq \left(\operatorname{Spr}\star \mathcal{E}_n\right)^{\operatorname{sgn}_W\otimes V} \simeq S_{\mathcal{N}_G}(\sgn_W\otimes V)\star\mathcal{E}_n.\]
 Combining the above isomorphisms, we get
\[
  \D\mathcal{E}_n'^V \simeq S_{\mathcal{N}_G}(\sgn_W\otimes V)\star\mathcal{E}_n.
\]
\end{proof}
\begin{proof}[Proof of the Proposition \ref{sec:whittaker-averaging-9}.] By Proposition~\ref{sec:char-sheav-full-1}, using that $\wedge^k\mathfrak{t} \simeq \sgn_W\otimes\wedge^{n-k}\mathfrak{t}$ (as representations of $W$), we have
  \[
\hc_!(\mathfrak{W}_k\star\mathcal{E}_n) \simeq \hc_!(\D\mathcal{E}_n'^{\wedge^{n-k}\mathfrak{t}}).
  \]
  By Proposition \ref{sec:four-deligne-transf-6},
  \[
\mathcal{E}_n'^{\wedge^{n-k}\mathfrak{t}} \simeq \mathfrak{W}_{n-k}\star\mathcal{E}_n'.
  \]
    Combining this with Proposition \ref{sec:whittaker-averaging-3},  we get
    \begin{multline*} \NN_{w_0}\star \hc_!(\mathfrak{W}_k\star\mathcal{E}_n) \dot{=} \NN_{w_0}\star \hc_!(\D\mathcal{E}_n'^{\wedge^{n-k}\mathfrak{t}})\simeq \\ \simeq \D\left(\NN_{w_0}\star\hc_!(\mathcal{E}_n'^{\wedge^{n-k}\mathfrak{t}})\right) \simeq \D (\NN_{w_0}\star \hc_!(\mathfrak{W}_{n-k}\star\mathcal{E}_n')).
    \end{multline*}
    Since, by definition, $\hc_!(\mathcal{E}_n')\simeq \varepsilon_n' \simeq \D(\varepsilon_n)$,  where in the last isomorphism we forget the $W$-equivariant structure, we get
    \[
      \D (\NN_{w_0}\star \hc_!(\mathfrak{W}_{n-k}\star\mathcal{E}_n'))\simeq \D (\NN_{w_0}\star\hc_!(\mathfrak{W}_{n-k}) \star\D\varepsilon_n ).
    \]
    So, passing to the limit (recall the Definition~\ref{sec:verdier-duality-1}), we get
       \[ \NN_{w_0}\star\hc_!(\mathfrak{W}_k\star \dd_G) \simeq \D'(\NN_{w_0}\star\hc_!(\mathfrak{W}_{n-k})\star\dd).
       \] Since, by Proposition \ref{braid_relations}, we have $\DD_{w_0}\star \NN_{w_0} \simeq \dd$, $\dd$ being a monoidal unit of $\cM$, we get the result.
     \end{proof}
      
     Combining Proposition \ref{sec:whittaker-averaging-9} with Theorem \ref{sec:whittaker-averaging-4} we get a concise description of objects representing the monodromic model of Hoch\-schild cohomology of Soergel bimodules:
     \begin{theorem}
       \label{sec:char-sheav-full-3}
       There is an isomorphism
       \[
         \mathfrak{E}_k \simeq p^{\dagger}\hc_!(\mathfrak{W}_k\star\dd_G)(2k-l(w_0)).
       \]
     \end{theorem}
     \begin{bremark}
      From Proposition \ref{sec:four-deligne-transf-6} one recovers a description of $\mathfrak{W}_k \star \dd_G$ as a projective system of Goresky-MacPherson extensions of explicit local systems on the set $G^{rss}$ of regular semisimple elements in $G$.  
     \end{bremark}
     We also make the following observation, of an independent interest. We have, by Corollary \ref{sec:four-deligne-transf-2}, \[\mathfrak{W}_n \simeq \IC_{\mathcal{N}_G^{reg}} = \underline{\Qlbar}_{\mathcal{N}_G^{reg}}\tw{\dim \mathcal{N}_G}.\]
     From Propositions \ref{sec:whittaker-averaging-9} and \ref{sec:char-sheav-full-1} we get 
     \begin{corollary}
       \label{sec:char-sheav-full-2}
      We have an isomorphism 
      \[
      \prolim p^{\dagger}\hc_!(\mathcal{E}_n^{\sgn_W}) \simeq p^{\dagger}\hc_!(\cc_{\mathcal{N}_G^{reg}}\tw{\dim\mathcal{N}_G}\star\dd_G)\simeq \DD_{w_0}^2.
      \]
     \end{corollary}
     Substituting $k = n$ in Theorem~\ref{sec:char-sheav-full-3}, we recover the computation done in type A in \cite{ghmn} for the object representing the highest Hochschild cohomology of Soergel bimodules.
     \begin{corollary}
       \label{sec:whittaker-averaging-6}
       We have an isomorphism
       \[ \mathfrak{E}_n = \mathcal{H}^{-n}_{w_0}(\mathbb{K}) \simeq \DD_{w_0}^2(2n-l(w_0)).
       \]
     \end{corollary}
     \begin{bremark} It is shown in \cite{beilinson2004tilting}, \cite{ghmn} that the convolution with $\DD_{w_0}^2$ (tensor product with the corresponding Rouquier complex $\Delta_{w_0}^2$) is the inverse to the Serre functor on the category $D^b(U\backslash G/B)$ (the homotopy category of Soergel bimoduels $\operatorname{Ho}(\SBim(W))$ in type $A$, respectively).
     \end{bremark}
     
     \begin{bremark} One may also recover Corollary \ref{sec:whittaker-averaging-6} as follows: since the functor $h\Ho(\Lambda)$ is t-exact and $\mathbb{K}$ is in its image, we get that $\F:= p^{\dagger}\hc_!(\IC_{n}\star\dd_G)$ is also in its image. Since $\DD_{w_0}\star\D'(\NN_{w_0}\star\F)(-l(w_0))$ represents the 0th Hochschild cohomology functor, by Yoneda lemma it is isomorphic to $R(-l(w_0))$. Hence, by Theorem \ref{sec:whittaker-averaging-4}, we get the result.
     \end{bremark}

\section{Jucys-Murphy filtrations in type A}
\label{sec:jm-filtr}
In this Section we work with $G = \GL_n$. 
\subsection{Jones-Ocneanu traces.}
We first recall the decategorified picture. Let $H_n$ be the Iwahori-Hecke algebra attached to the Weyl group of $G$: $H_n$ is a unital algebra over $\mathbb{Q}(v)$ generated by the elements
\[t_s, s \in \Sigma = \{s_1, \dots, s_{n-1}\}\] satisfying the braid relations along with the quadratic relations \[t_{s}^2 = (v-v^{-1})t_s+1.\]

It is well-known that $H_n$ has a basis $t_w, w \in W,$ with $t_w = t_{r_1}\dots t_{r_{k}}, r_i \in \Sigma,$ for a reduced expression $w = r_1 \dots r_k$.

Write $H_0 = H_1 = \mathbb{Q}(v,a)$, and let $v^2 = q$, where $q$ is treated as a formal variable.

Write $\iota: H_{n-1} \to H_{n}$ for the embedding sending $t_i$ to $t_i$ for $i = 1,\dots, n-2$.

By the results of Ocneanu and Jones, see \cite{jones}, there is a system of traces $\operatorname{Tr}_n:H_{n} \to \mathbb{Q}(a,v)$, uniquely determined by the following properties:
\begin{enumerate}[label=(\arabic*)]
\item \(\operatorname{Tr}_n(ab) = \operatorname{Tr}_n(ba), a,b \in H_{n},\)
\item \(\operatorname{Tr}_n(\iota(a)) = \dfrac{1+a}{1-q}\operatorname{Tr}_{n-1}(a), a \in H_{n-1},\)
\item \(\operatorname{Tr}_n(\iota(a)t_{n-1})=-v^{-1}\operatorname{Tr}_{n-1}\iota(a)\)
\item \(\operatorname{Tr}_0(1) = 1\).
\end{enumerate} We have a homomorphism $\phi:\mathbb{Z}[v,v^{-1}][\operatorname{Br}_n] \to H_n$, sending the generator $\sigma_w$ to $t_w$.

It follows from the trace-like property of $\operatorname{Tr}_n$ that it is completely determined by the coefficients $W_\lambda$ in the decomposition
\[ \operatorname{Tr}_n = \sum_{\lambda}W_{\lambda}\operatorname{Tr}_\lambda,
\] where $\lambda$ runs through the Young diagrams with $n$ cells that label the irreducible representations of $H_n$, and $\operatorname{Tr}_\lambda$ stands for the character of the corresponding representations. The standard labeling is chosen such that the rank one representation where $t_s$ acts by $v$ corresponds to the diagram with one row. We have the following
\begin{proposition}[\cite{wenzl}, \cite{jones}]\label{sec:jones-ocne-trac} The weights $W_\lambda$ are given by
  \[ W_\lambda = q^{n'(\lambda)} \prod_{\square \in \lambda}\dfrac{1+aq^{-c(\square)}}{1-q^{h(\square)}}
  \]
\end{proposition} Here the products runs through the cells $\square \in \lambda$, and the following conventions for the diagram parameters are used. The diagram is assumed to be positioned in the lower-right quadrant of the plane, all cells given non-negative coordinates as on the Figure \ref{fig:1}. For a cell $\square = (i,j)$, $c(\square) = i-j$, the content, $h(\square)$ is the hook length, and $n'(\lambda) = \sum_i(i-1)\lambda'_i$, where $\lambda'_i$ are lengths of the rows of the transposed diagram.
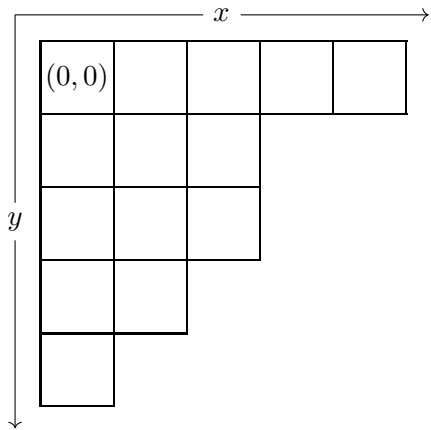
\begin{figure}[h]
    \caption{Coordinates of cells in the Young diagram}
    \label{fig:1}
    \centering
    \stackon[-147pt]
    {
      \ytableausetup{boxsize=2.5em} \ytableaushort{{(0,0)},\none,\none,\none,\none}*{5,3,3,2,1} } {
      \begin{tikzpicture}
        \draw (0,0) -- (2.5,0) node[right] {$x$};
        \draw[->] (3,0) -- (5.5,0);
        \draw (0,0) -- (0,-2.5) node[below] {$y$};
        \draw[->] (0,-3) -- (0,-5.5);
      \end{tikzpicture}
    }
\end{figure}

Write \[j_0 = 1, j_k = s_kj_{k-1}s_k, k = 0, \dots, n-1.\] Elements $j_i$ are known as Jucys-Murphy elements in the Hecke algebra $H_n$. It is easy to check that $j_kj_l = j_lj_k$ for all $k, l$. Let $P$ be any symmetric polynomial in $n$ variables. Then it is well-known that an element $P(j_0^{-1},\dots,j_{n-1}^{-1})$ acts on the representation of $H_n$ labeled by $\lambda$ as a constant $P(q^{c_\lambda})$, where $c_\lambda$ is a vector of contents of cells of $\lambda$ (defined up to permutation). See \cite{isaev2005representations}, \cite{okounkov1996new}.

Let $\operatorname{Tr}^{(k)}_n \in \mathbb{Q}(v)$ stand for the coefficient of $\operatorname{Tr}_n$ near $a^k$. From this discussion and Proposition \ref{sec:jones-ocne-trac} we get the following
\begin{corollary}
\label{sec:jucys-murphy-filtr-2} $\operatorname{Tr}^{(k)}_n$ and $\operatorname{Tr}^{(0)}_n$ are related by the following
  \[ \operatorname{Tr}_n^{(k)}(x) = \operatorname{Tr}_n^{(0)}(xE_k(j_0^{-1},\dots,j_{n-1}^{-1})),
  \] where $E_k$ stands for the $k$-th elementary symmetric polynomial.
\end{corollary}

Let $\mathfrak{j}_k \in \operatorname{Br}_n$ stand for the Jucys-Murphy braids defined by
\[{\mathfrak{j}}_0 = 1, \mathfrak{j}_k = \sigma_k{\mathfrak{j}}_{k-1}\sigma_k, k = 1, \dots, n-1.\] Note that the Jucys-Murphy elements $\mathfrak{j}_k$ are images of braid group elements $\mathfrak{j}_k$ under the homomorphism $\mathbb{Z}[\operatorname{Br}_n]\to H_n, \sigma_s \mapsto t_s$. It is easy to check that the elements ${\mathfrak{j}}_k$ also commute with each other in $\operatorname{Br}_n$.

\subsection{Jucys-Murphy filtrations.}
For a subset $\mu \subset \{0,\dots,n-1\}$ write $\mathfrak{j}^\mu = \prod_{i \in \mu}\mathfrak{j}_i$ and let \[\mathfrak{J}^\mu = h\operatorname{Ho}\Lambda(F_{\mathfrak{j}^\mu}) \in \cM,\] the geometric categorification of the product of Jucys-Murphy braids.

For an objects $X, \{C_i\}_{i = 1}^k$ of a triangulated category, we write $X \in \langle C_1, C_2, ..., C_k\rangle$ if there exists a sequence of objects $\{X_i\}_{i=1}^k, X_1 = X, X_k = C_k$, such that for all $i < k$, there is a distinguished triangle
\[ C_i \to X_i \to X_{i+1} \to C_i[1].\]

The main result of this subsection is the following
\begin{theorem}
\label{sec:jucys-murphy-filtr-1} The objects representing Hochschild cohomology in $\cM$ satisfy
  \[ \mathfrak{E}_k(-2k+l(w_0))\in\langle\mathfrak{J}^\mu\rangle_{|\mu| = k}
  \] The indices are ordered so that $\mu = (\mu_1, \dots, \mu_k) \geq \mu' = (\mu'_1,\dots,\mu'_k)$ if $\mu_1 \leq \mu_1', \dots, \mu_k \leq \mu_k'.$
\end{theorem}

Fix a braid $\beta \in \operatorname{Br}_n$. Khovanov-Rozansky homology $\operatorname{HHH}(\beta)$ is a triply graded vector space, with grading pieces $\operatorname{HHH}^{k,i,j}(\beta) = \operatorname{HH}^k(F_\beta)$, where $k$ stands for the Hochschild grading, $i$ for the cohomological grading on the Rouquier complex $F_\beta$ and $j$ for the inner grading on the bimodules. It is proved in \cite{khovanov2007triply} that, taking the Euler characteristic with respect to the second granding, one recovers the Jones-Ocneanu trace as the Hilbert-Poincar\'{e} series
\[ \operatorname{Tr}_n = \sum_{k,i,j}(-1)^i\dim \operatorname{HHH}^{k,i,j}(\beta)(-2k)a^kq^j.
\]

By Theorem \ref{sec:comp-with-soerg-2}, we have
\[ \operatorname{HH}^k(F_\beta) = \underline{\Hom(\mathcal{H}_{w_0}^{-k}(\mathbb{K}),h\operatorname{Ho}(\Lambda)(F_\beta))^{f,\Frinv}}.
\]

From the additivity of Euler characteristic in distinguished triangles, using the Theorem \ref{sec:jucys-murphy-filtr-1}, we obtain
\[ \operatorname{Tr}_n^{(k)} =\sum_{|\mu|=k}\sum_{i,j}(-1)^i\dim \operatorname{HH}^{0,i,j}(F_\beta F_{\mathfrak{j}^\mu}^{-1})q^j,
\] which is exactly the categorification of Corollary \ref{sec:jucys-murphy-filtr-2}.
\begin{proof}[Proof of Theorem \ref{sec:jucys-murphy-filtr-1}]

  For a parabolic $P \subset G$ with unipotent radical $U_P$ and Levi $L_P$, write
  \[ \mathcal{Y}_P = \dfrac{G/U_P\times G/U_P}{L_P},
  \] where $L_P$ acts on $G/U_P \times G/U_P$ diagonally on the right. This $L_P$-action commutes with the diagonal $G$-action on the left. Note that $D_G^b(\mathcal{Y}_B)\simeq D(\mathcal{Y})$, so the spaces $\mathcal{Y}_P$ are symmetric parabolic generalizations of $\mathcal{Y}$. Note also that $D^b_G(\mathcal{Y}_G) = D^b_{\Ad G}(G)$.

  For parabolics $P \subset Q$ we have functors $\hc_{P}^Q, \chi_{P}^Q$, defined by Lusztig in \cite{lpc1}. We recall the definition. Let \[\tilde{Z}_{P, Q} = \{(xU_Q, yU_Q, zU_P): x^{-1}z \in Q\},\] and let $Z_{P,Q}$ be its quotient by $L_Q \times L_P$ action on the right. We have maps
  \[
    \begin{tikzcd} & Z_{P,Q} \arrow[ld, "{f_{P,Q}}"'] \arrow[rd, "{g_{P,Q}}"] & \\ \mathcal{Y}_Q & & \mathcal{Y}_P
    \end{tikzcd}
  \]
  \[ f_{P,Q}(xU_Q, yU_Q, zU_P) = (xU_Q, yU_Q),
  \]
  \[ g_{P,Q}(xU_Q, yU_Q, zU_P) = (zU_P, yx^{-1}zU_P).
  \] Write
  \[ \chi_{P}^{Q} = f_{P,Q!}g_{P,Q}^*[\dim G/P-\dim G/Q],
  \]
  \[ \hc_P^Q=g_{P,Q!}f_{P,Q}^*[\dim G/Q - \dim G/P].
  \] In our previous notations, we have $\hc_! = \hc_B^G, \chi = \chi_B^G$. It was proved in loc. cit. that if $P \subset Q \subset R$, we have
  \[ \chi_Q^R\chi_P^Q \simeq \chi_{P}^R, \hc^Q_P\hc^R_Q \simeq \hc^R_P.
  \] The functors $\hc_{P}^Q$ are monoidal with respect to convolution.
  
  Finally, let $\Delta_P \subset \mathcal{Y}_P, \Delta_P = \{(xU_P, yU_P)L: xP = yP\}$. Then $D_G^b(\Delta_P) \simeq D_{\Ad L_P}^b(L_P)$, and the functor $\chi_P^G$ (respectively $\hc^P_B$ for $P \supset B$), restricted to sheaves supported on $\Delta_P$, coincides with the functor $\operatorname{Ind}_{L_P}^G$ (respectively $\res^{L_P}_{B\cap L_P})$.

 Let $P_G(\mathcal{N}_G)$ be the category of $G$-equivariant mixed perverse sheaves on the unipotent variety of $G$. Recall the functor $S_{\mathcal{N}_G}:\operatorname{Rep}(W)\to P_G(\mathcal{N}_G/G)$ from Section \ref{sec:four-deligne-transf-4}. Let $H_G:P_G(\mathcal{N}_G)\to \operatorname{Rep}(W)$ be the bi-adjoint functor $H_G = \Hom(\operatorname{Spr}, -).$ We will need the following compatibility result.
 \begin{lemma}
  \label{sec:jucys-murphy-filtr-5} Let $P\supset B$ be a standard parabolic subgroup with the Levi quotient $L$. Assume that $P$ corresponds to a subset $J \subset \Sigma$ of the set of simple reflections, and let $W_L\subset W$ be the corresponding parabolic subgroup.
  \begin{enumerate}[label=(\arabic*)]
  \item \label{item:15} There are natural isomorphisms
    \[ \res^G_{P}\circ S_{\mathcal{N}_G} \simeq S_{ \mathcal{N}_L }\circ \res^W_{W_L}, \res^W_{W_L}\circ H_G \simeq H_L\circ \res^G_{P}.
    \]
  \item\label{item:16} There are natural isomorphisms
    \[ \ind^G_{P}\circ S_{ \mathcal{N}_L } \simeq S_{ \mathcal{N}_G }\circ \ind^W_{W_L}, \ind^W_{W_L}\circ H_L \simeq H_G\circ \ind^G_{P}.
    \]
  \item\label{item:17} Let $L_1, L_2$ be two reductive groups, let $W_{L_1}, W_{L_2}$ be their Weyl groups. Let $X \in \operatorname{Rep}(W_{L_1}), Y \in \operatorname{Rep}(W_{L_2}).$ We have
    \[ S_{\mathcal{N}_{ L_1\times L_2 }} (X \boxtimes Y)\simeq S_{\mathcal{N}_{ L_1 }}(X)\boxtimes S_{\mathcal{N}_{ L_2 }}(Y).
    \]
  \end{enumerate}
 \end{lemma}
 \begin{proof} We prove properties \ref{item:15} and \ref{item:17}, \ref{item:16} follows by adjunction. Recall that the functor $S_{\mathcal{N}_G} = S_{\mathcal{N}_G}'\otimes \operatorname{sgn}_W$ is defined from (the sign twist of the) restriction of the functor $S'_G$. It is well known that, restricted to the open subset of regular semisimple elements $G^{rss}\subset G$, parabolic restriction functor $\res^G_P|_{G^{rss}}$ coincides, up to a homological shift and Tate twist, with the regular restriction along the embedding $L^{rss}\to G^{rss}$, see e.g. \cite{ginzburg1993induction}, Theorem 4.1 (ii). For the regular restriction the claim is immediate, and the Lemma follows.
 \end{proof}
 Let $P_k \subset G$ be a maximal parabolic subgroup of operators preserving a $k$-dimensional subspace, and let $L_k$ be its Levi. We assume that $P_k \supset B$ and that this parabolic corresponds to the subset $J_k := \{s_1, \dots, \hat{s}_k, \dots, s_{n-1}\} \subset \Sigma$. Note that $L_k \simeq \GL_k \times \GL_{n-k}$.

 \begin{lemma}
   \label{sec:jucys-murphy-filtr-6}
   We have an isomorphism
   \[
     \hc_{P_k}^G(\dd_{\GL_n})\simeq \dd_{\GL_k}\boxtimes\dd_{\GL_{n-k}}.
   \]
 \end{lemma}
 \begin{proof}
   Let $\dd_{k}:=\dd_{\GL_k}\boxtimes\dd_{\GL_{n-k}}$.
   
   By definition of $\dd_{G}$ for an arbitrary reductive $G$, it is easy to see that
   \[
     \res^G_{P_k}(\dd_{G}) \simeq \dd_{k},
   \]
   and
   \[\hc_B^P(\dd_k)\simeq \res^L_{B\cap L_k}(\dd_k) \simeq \dd.\] 
   Consider the canonical morphism $r:\hc^G_P\dd_G \to \res^G_P\dd_G$. Applying $\hc^P_B$ to it, we get an isomorphism
   \[
     \dd \simeq \hc^G_B\dd_G \simeq \hc^{P_k}_B\hc^G_{P_k}\dd_G \simeq \hc^{P_k}_B\res^G_{P_k}\dd_k = \res^{L_k}_{B \cap L_K}\res^G_{P_k}\dd_k \simeq \dd_G.
   \]
   By a straightforward parabolic analogue of Lemma \ref{sec:four-deligne-transf-5}, one sees that the functor $\chi_B^{P_k}\hc^{P_k}_B$ contains an identity functor as a direct summand. It follows that, since $\hc^{P_k}_Br$ is an isomorphism, $r$ must also be an isomorphism, since $\hc^{P_k}_B$ sends its cone to $0$.  
 \end{proof}

 Summarizing the discussion above, we get the following expression.
  \begin{lemma}
    \label{sec:jucys-murphy-filtr-4} There is an isomorphism
    \[ \hc_!(\mathfrak{W}_k \star \dd_G) \simeq \hc_B^{P_k}\hc_{P_k}^G\chi_{P_k}^G((\mathfrak{W}^{\GL_{k}}_{k}\boxtimes \mathfrak{W}_0^{\GL_{n-k}})\star\hc_P^G(\dd_G)).
    \]
  \end{lemma}
  \begin{proof}
 We have $\h$ isomorphic to the $n$-dimensional permutation representation of $W = S_n$,  $W_{L_k}\simeq S_k\times S_{n-k}$ and, for $k \geq 1$, \[\ind_{W_{L_k}}^W(\operatorname{sgn}_{S_k}\boxtimes\operatorname{triv}_{S_{n-k}}) \simeq \wedge^k\h.\]

 Since, by definition, $\mathfrak{W}_{k} = S_{\mathcal{N}_G}(\wedge^k\h)$,  
 we get, from Lemma \ref{sec:jucys-murphy-filtr-5} \ref{item:16} and \ref{item:17}, \[\mathfrak{W}_{k} \simeq \operatorname{Ind}_{L_{k}}^G(\mathfrak{W}^{\GL_{k}}_{k}\boxtimes \mathfrak{W}_0^{\GL_{n-k}}) \simeq \chi_{P_k}^G(\mathfrak{W}^{\GL_{k}}_{k}\boxtimes \mathfrak{W}_0^{\GL_{n-k}}).\]
 Applying $\hc = \hc_P^G\hc_B^P$ to both sides, and using the fact that \[\chi_{P_k}^G(\F \star \hc_{P_k}^G(\mathcal{A})) \simeq \chi_{P_k}^G(\F)\star\mathcal{A}\] we get the result.
  \end{proof}
  Let $w_0^{(k)}$ be the longest element of the subgroup $S_k \subset S_n = W$ generated by $s_1, \dots, s_{k-1}$. By Proposition \ref{sec:whittaker-averaging-9} and Lemma \ref{sec:jucys-murphy-filtr-6}, we have
  \begin{equation}
\label{eq:4} \hc_B^P((\mathfrak{W}^{\operatorname{GL}_{k}}_{k}\boxtimes \mathfrak{W}_0^{\operatorname{GL_{n-k}}})\star\hc_P^G(\dd_G)) \simeq \DD_{w_0^{(k)}}^2.
  \end{equation}
  
  The following lemma is a straightforward consequence of standard distinguished triangles for 6 functors for constructible sheaves.
  \begin{lemma}
    \label{shrieck_filtr} Let $X$ be a variety stratified by locally closed subvarieties $\{S_t\}_{t=1}^n$, and let $j_t: S_t \to X$ be the corresponding locally-closed embeddings. Assume that $S_k \subset \overline{S}_l$ implies $k > l$. Then for every $\mathcal{F} \in D^b(X)$, $\mathcal{F} \in \langle j_{t!}j^*_t\mathcal{F}\rangle_{t=1}^n$.
  \end{lemma}
  \begin{proof} Let $X_k = X \backslash (\cup_{i > k}S_i), k = 1, \dots ,n$. The Lemma is obvious for the stratification consisting of the single stratum. Write $r_k:X_k \to X_{k+1}, s_k:X_k \to X$ for the corresponding open embeddings. Assume that it is known for $X$ replaced with $X_k, k < n$, namely that $s_k^*\F \in \langle j^{(k)}_{t!}j^{*}_t\F\rangle_{t=1}^k$, where we denote by $j_t^{(k)}$ an embedding $S_t \to X_k, t \leq k$. We have a standard distinguished triangle corresponding to complementary open and closed embeddings $r_k, j^{(k+1)}_{k+1}$
   \[ r_{k!}r_k^*s_{k+1}^*\F \to s_{k+1}^*\F \to j^{(k+1)}_{k+1*}j^{(k+1)*}_{k+1}s_{k+1}^*\F \to r_{k!}r_k^*s_{k+1}^*\F[1].
   \] Note that $s_{k+1} \circ r_k = s_{k}, s_{k+1}\circ j_t^{(k+1)} = j_t$. It follows that
   \[ s_{k+1}^*\F \in \langle r_{k!}j^{(k)}_{1!}j^{*}_1\F, \dots, r_{k!}j^{(k)}_{t!}j^{*}_t\F, j^{(k+1)}_{k+1*}j^{*}_{k+1}\F \rangle.
   \] Since $r_k\circ j_t^{(k)} = j_t^{(k+1)}$ and $j^{(k+1)}_{k+1*} = j^{(k+1)}_{k+1!}$,
\[s_{k+1}^*\F \in \langle j^{(k+1)}_{t!}j^{*}_t\F\rangle_{t=1}^{k+1}\] which completes the induction step.
  \end{proof} Let $P \supset B$ be the parabolic corresponding to the subset $J \subset \Sigma$. Let $W_J \subset W$ be the subgroup generated by $J$, and let $W^J$ be the set of minimal length representatives of $W/W_J$. We have the following generalization of \cite{lusztig2013truncated}, Proposition 2.6 (see also \cite{grojnowski1992character} for the corresponding statement on the level of the Grothendieck group).
  \begin{lemma}\label{sec:jucys-murphy-filtr-3} For $\F \in D^b_G(\mathcal{Y}_P)$ such that $p^{\dagger}\hc_B^P(\F) \in \omega\M$, we have \[\hc_!\chi_P^G(\F) \in \langle\DD_w\star\hc_B^P(\F)\star\DD_{w^{-1}}\rangle_{w \in W^J},\] $w$ taken in some non-increasing order with respect to the Bruhat partial ordering.
  \end{lemma}
  \begin{proof} By proper base change, we have $\hc_!\chi_P^G \simeq \alpha_!\beta^*,$ where $\alpha$ and $\beta$ are from the following diagram:
    \[
      \begin{tikzcd} & G\times G/P \times G/B \arrow[ld, "\beta"'] \arrow[rd, "\alpha"] & \\ \mathcal{Y}_P & & \mathcal{Y_B}
      \end{tikzcd}
    \] $\alpha(g, xP, yB) = (yU, gyU)T, \beta(g, xP, yB) = (xU_P, gxU_P)L_P$. The filtration in the Lemma comes from the filtration of $G \times G/P \times G/B$ by the locally closed subvarieties $X_w = \{(g, xP, yB), xP \text{ in relation }w\text{ to }yB\},$ $w \in W^J$, using Lemma \ref{shrieck_filtr}.
  \end{proof} Combining Lemmas \ref{sec:jucys-murphy-filtr-4}, \ref{sec:jucys-murphy-filtr-3} with \eqref{eq:4}, we get
  \begin{corollary}
    \label{sec:jucys-murphy-filtr} $\hc_!(\mathfrak{W}_k \star \dd_G) \in \langle\DD_w\star\DD_{w_0^{(k)}}^2\star\DD_{w^{-1}}\rangle_{w \in W^{J_k}}$.
  \end{corollary} It remains to relate the braids of the form $\sigma_w\sigma_{w_0^{(k)}}^2\sigma_{w^{-1}}, w \in W^{J_k}$ to the products of the Juscys-Murphy braids $\mathfrak{j}^{\mu}, |\mu|=k$. Let us regard the symmetric group $W$ as a group of permutations of the set $\{0, \dots, n-1 \}$, with simple reflections $s_i = (i-1\ i).$ In this way, $W_{J_k}$ is identified with the subgroup $S_k\times S_{n-k}$ of permutations preserving the subset $\{0,\dots,k-1\}$. We have the following combinatorial Lemma, verified by the straightforward computation.
  \begin{lemma} The braid $\mathfrak{j}^{\mu}, |\mu|=k$ is equal to the braid $\sigma_w\sigma_{w_0^{(k)}}^2\sigma_{w^{-1}}$, where $w$ is the minimal length representative of the coset of permutations in $S_n/S_k \times S_{n-k}$ sending $\{0,\dots,k-1\}$ to $\mu$.
  \end{lemma}
\end{proof}
\begin{bremark}
  
  Note that the braids $\mathfrak{j}^{\mu}$ are of the form $\sigma_w\sigma_{w^{-1}}$ for some $w \in W$, where we denote the lift of $w$ to the braid group with $\sigma_w$. We then have $\sigma_{w_0}^{-1}\sigma_w\sigma_{w^{-1}} = \sigma_{w^{-1}w_0}^{-1}\sigma_{w^{-1}}$. So representations of these braids in the monodromic Hecke categories can be shown to lie in $\Perv_{w_0}$: the object of the form $\NN_{w_0}\star\DD_{w}\star\DD_{w^{-1}}\simeq \NN_{w_0w}\star\DD_{w^{-1}}$ is in $\Perv'$ because the convolution with $\NN_{w_0w}$ is t-exact from the right, and convolution with $\DD_w$ is t-exact from the left. See, e.g., Lemma 7.7 from \cite{achar2019mixed}, whose proof straightforwardly adapts to our situation. Cf. also Section 5.1 of \cite{arkhipov2009perverse}.
\end{bremark}
\appendix
\section{Examples}
\subsection{Type $\operatorname{A}_1$.} We have $R = \Qlbar[x], s(x) = -x.$ Write $R \otimes R = \Qlbar[x,y]$ \[K_{\mathbb{S}}= B_s(-2) \xrightarrow{y - x} \underline{B_s}, \nabla_s = R(-1) \to \underline{B_s}, \Delta_s = \underline{B_s} \to R(1).
\] We have \[\nabla_s K_{\mathbb{S}} = B_s(-1) \xrightarrow{x + y} \underline{B_s(1)},\] and a morphism of complexes $\Delta_s[1](-1)\to \nabla_sK_{\mathbb{S}}$
\[
\begin{tikzcd} B_s(-1) \arrow[d, "\operatorname{id}"'] \arrow[r, "a\otimes b\mapsto ab"] & \underline{R} \arrow[d, "1 \mapsto x + y"] \\ B_s(-1) \arrow[r, "-x-y"'] & \underline{B_s(1)}
\end{tikzcd}
\] which gives a distinguished triangle
\[ \Delta_s[1](-1) \to \nabla_s K_{\mathbb{S}} \to \nabla_s(1) \to \Delta_s[2](-1).
\] From the long exact sequence of cohomology we get $\mathcal{H}'^0(\nabla_s K_{\mathbb{S}}) = \nabla_s(1),\mathcal{H}'^{-1}(\nabla_s K_{\mathbb{S}}) = \Delta_s(-1),$ (note that $\Delta_s(-1)$ and $\nabla_s(1)$ are in the heart of the t-structure, by definition) and, finally,
\[ \mathcal{H}^0_{w_0}(K_{\mathbb{S}}) = R(1), \mathcal{H}^{-1}_{w_0}(K_{\mathbb{S}}) = \Delta_s^2(-1).
\]
\subsection{Type $\operatorname{A}_2$.} We omit most of the differentials from notations. Let $R = \Qlbar[x_1, x_2], R \otimes R = \Qlbar[x_1, x_2, y_1, y_2]$. Write $B_i = B_{s_i}, B_{ij} = B_{i}B_{j}, B_{121} = B_{s_1s_2s_1} = B_{s_2s_1s_2}$.

We have $B_1B_2B_1 = B_{121}\oplus B_1, B_2B_1B_2 = B_{121}\oplus B_2$.
\[ \Delta_{w_0} = \underline{B_{121}} \to \color{blue}{B_{12}(1)}\color{black}\oplus\color{red}{B_{21}(1)}\color{black} \to \color{blue}B_1(2) \oplus B_2(2)\color{black} \to \color{blue}R(3),
\]
\[ \nabla_{w_0} = \color{red}{R(-3)}\color{black} \to \color{red}{B_1(-2) \oplus B_2(-2)}\color{black} \to \color{blue}{B_{21}(-1)}\color{black}\oplus \color{red}{B_{12}(-1)}\color{black} \to \color{black}\underline{B_{121}},
\]
\[ K_{\mathbb{S}} = B_{121}(-4) \xrightarrow{(y_2-x_2) \oplus (x_1 - y_1)} B_{121}(-2)^{\oplus 2} \xrightarrow{(y_1-x_1)\oplus (y_2 - x_2)} \underline{B_{121}},
\]
\[ \nabla_{w_0}K_{\mathbb{S}} = B_{121}(-1) \xrightarrow{(y_2+x_1) \oplus (-x_1 - y_1)} \color{blue}{B_{121}(1)}\color{black}{\oplus}\color{red}{B_{121}(1)}\color{black} \xrightarrow{(y_1 + x_2)\oplus (y_2 + x_1)} \underline{B_{121}(3)}.
\] We have maps
\[ \Delta_{w_0}[3](-1)\to\nabla_{w_0}K_{\mathbb{S}}, \nabla_{w_0}K_{\mathbb{S}} \to \nabla_{w_0}(3).
\] Taking the cone twice, we see that $\mathcal{H}'^0(\nabla_{w_0} K_{\mathbb{S}}) = \nabla_{w_0}(3),\mathcal{H}'^{-2}(\nabla_{w_0} K_{\mathbb{S}}) = \Delta_{w_0}(-1)$, and $\mathcal{H}'^{-1}$ is an extension of shifted Jucys-Murphy complexes
\[ \nabla_{2}\Delta_{12} = \color{blue}{B_{12}(-1) \to \underline{B_{121} \oplus B_2 \oplus B_1} \to B_{21}(1)\oplus R(1)}
\] and
\[ \nabla_{12}\Delta_1 = \color{red}{B_{21}(-1) \oplus R(-1) \to \underline{B_{121} \oplus B_2 \oplus B_1} \to B_{12}(1)}.
\] The terms appearing in the corresponding complexes are given their corresponding color.

We get
\[\mathcal{H}^0_{w_0}(K_{\mathbb{S}}) = R(3), \mathcal{H}^{-2}_{w_0}(K_{\mathbb{S}}) = \Delta_{w_0}^2(-1),
\] and $\mathcal{H}^{-1}_{w_0}$ is an extension of $\Delta_{21}\Delta_{12}(1)$ and $\Delta_1^2(1)$.
\begin{bremark}

  Note that the grading shifts have signs that are opposite to those in Theorem \ref{sec:char-sheav-full-3}, since the action of the Frobenius is inverted when comparing monodromic categories with Soergel bimodules, cf. Proposition \ref{sec:comp-with-soerg-4}.

\end{bremark}
\bibliography{hhh}
\bibliographystyle{plain}
\Addresses
\end{document}